\documentclass[preprint,12pt,3p]{elsarticle}
\usepackage{color}
\usepackage[colorlinks=true]{hyperref}
\usepackage{amsmath,amssymb,amsthm,amscd}
\numberwithin{equation}{section}

\newtheorem{theorem}{Theorem}[section]

\newtheorem{proposition}{Proposition}[section]

\newtheorem{definition}{Definition}[section]
\newtheorem{corollary}{Corollary}[section]
\newtheorem{remark}{Remark}[section]

\journal{Elsevier}

\begin{document}

\begin{frontmatter}
\title{Existence and multiplicity of solutions to Dirichlet problem for semilinear subelliptic equation with a free perturbation\tnoteref{label0}}
\tnotetext[label0]{This work is supported by National Natural Science Foundation of China (Grants No. 11631011 and 11626251) and China Postdoctoral Science Foundation (Grant No. 2020M672398).}
\author[label1]{Hua Chen}
\ead{chenhua@whu.edu.cn}
\author[label2]{Hong-Ge Chen}
\ead{hongge\_chen@whu.edu.cn}
\author[label1]{Xin-Rui Yuan}
\ead{yuanxinrui@whu.edu.cn}
\address[label1]{School of Mathematics and Statistics, Wuhan University, Wuhan 430072, China}
\address[label2]{Wuhan Institute of Physics and Mathematics\\ Innovation Academy for Precision Measurement Science and Technology\\ Chinese Academy of Sciences, Wuhan 430071, China}
\begin{abstract}
This paper is concerned with existence and multiplicity results for the semilinear
subelliptic equation with free perturbation term. By using the degenerate Rellich–Kondrachov compact embedding theorem, precise lower bound estimates of Dirichlet eigenvalues for the finitely degenerate elliptic operator and minimax method, we obtain the existence and multiplicity of weak solutions for the problem.

\end{abstract}
\end{frontmatter}

\section{Introduction and Main Results}

Let $X=(X_1, X_{2},\cdots,X_m)$ be a system of real smooth vector fields defined on an open domain $W$ in $\mathbb{R}^n~ (n\geq 2)$ and satisfy the following H\"ormander's condition (H):\par
(H): $X_{1},X_{2},\ldots,X_{m}$ together with their commutators up to a of length at most $Q$ span the tangent space at each point of $W$. \par

Consider the following semilinear subelliptic Dirichlet problem with a free perturbation,
\begin{equation}\label{problem1-1}
\left\{
      \begin{array}{cc}
      -\triangle_{X}u=f(x,u)+g(x) & \mbox{in}~\Omega, \\[2mm]
      u=0           & \mbox{on}~\partial\Omega,
      \end{array}
 \right.
\end{equation}
where $\Omega\subset\subset W$ is a bounded connected
  open domain with $C^{\infty}$ boundary and the boundary $\partial\Omega$ is assumed to be
  non-characteristic for $X$, $\triangle_{X}:=-\sum_{i=1}^{m}X_{i}^{*}X_{i}$ is a H\"{o}rmander type operator, $g\in L^2(\Omega)$ and $f\in C(\overline{\Omega}\times\mathbb{R})$ with following assumptions:
\begin{itemize}
\item [$(f_{1})$] $f(x,0)=0$, and $\lim_{u\to 0}{\frac{f(x,u)}{u}}= 0$ uniformly in $x\in\overline{\Omega}$.
\item [$(f_2)$]  There exist $2<p<\frac{2\tilde{\nu}}{\tilde{\nu}-2}$ and $C>0$ such that
\[|f(x,u)| \leq C(1+|u|^{p-1})\]
for all $x\in\overline{\Omega}$, $u\in\mathbb{R}$.
\item [$(f_3)$] There exist $q>2$ and $R_0>0$ such that
\[0<qF(x,u)\leq f(x,u)u\]
for all $x\in\overline{\Omega}$ and $|{u}|\geq R_0$.
\item [$(f_4)$]  $f$ is odd with respect to variable $u$:~$f(x,-u)=-f(x,u)$.
\end{itemize}
Here $F(x,u)=\int_{0}^{u}f(x,v)dv$ is the primitive of $f(x,u)$, and $\tilde{\nu}\geq 3$ is the generalized M\'{e}tivier index which is defined in the Definition \ref{def2-2} below.\par

 The H\"{o}rmander type operator $\triangle_{X}$ (also called the finitely degenerate elliptic operator) is an important class of degenerate elliptic operator with smooth coefficients, which has been intensively studied in the late 1960s and is still an active research field. Here we only briefly recall some remarkable results, and one can refer to \cite{Bramanti2014} for more details. In 1967, H\"{o}rmander \cite{hormander1967} proved that, if the H\"{o}rmander's condition is satisfied, then $\triangle_{X}$ is hypoelliptic and admits the subelliptic estimates. Therefore,  $\triangle_{X}$ is also known as the subelliptic operator. Then, the maximum principle and Harnack inequality of $\triangle_{X}$ have been studied by Bony in \cite{Bony1969}. Later in 1976, Rothschild and Stein \cite{stein1976} established the sharp regularity estimates of $\triangle_{X}$ by their celebrated lifting and approximating theory. Meanwhile, M\'{e}tivier \cite{Metivier1976} investigated the eigenvalue problem of $\triangle_{X}$ under the M\'{e}tivier's condition. Furthermore, the H\"{o}rmander's condition allows us to define a Carnot-Carath\'{e}odory metric associated with vector fields, which plays important roles both in sub-Riemannian geometry and PDEs (cf. \cite{Montgomery2002}). Then, Nagel, Stein and Wainger studied the balls with the Carnot-Carath\'{e}odory metric in \cite{Stein1985}, and they gave the precise estimate for their volume. After that,
 the Poincar\'{e} inequality, Sobolev embedding theorem, estimation of Green kernel and heat kernel have also been studied by Jerison, Sanchez-Calle, Capogna, Danielli, Garofalo, Yung, etc. One can refer to \cite{Brandolini2010,Capogna1993,Jersion1986duke,Jerison1986,Sanchezcalle1984,Yung2015} as well as the reference therein.\par

When $X=(\partial_{x_{1}},\ldots,\partial_{x_{n}})$, $\triangle_{X}$ recovers to the classical Laplacian $\triangle$. The existence and multiplicity of weak solutions for the semilinear elliptic Dirichlet problem \eqref{problem1-1} with a free perturbation have been object of a very careful analysis by Bahri and Berestychi in \cite{Bahri1981}, Struwe in \cite{Struwe2000}, Dong and Li in \cite{Dong1982}, and Rabinowitz in \cite{Rabinowitz1982,Rabinowitz1986} via techniques of classical critical point theory. Moreover, these results have been improved by Bahri and Lions in \cite{Bahri1988,Bahri1992} via a  Morse-Index type technique around 1990, and also generalized to quasilinear elliptic equation in \cite{Sqassina2006}.\par

In the degenerate case, under the M\'etivier's condition, the nonlinear subelliptic equations without free perturbation terms have been studied by Xu, Zuily and Venkatesha Murthy in \cite{Murthy2008,Xu1990,Xu1995,Xu1996,Xu1997}. However we can find a lot of degenerate vector fields (e.g. the Grushin type vector fields) in which the M\'etivier's condition will be not satisfied. In this paper, we shall consider the problems no matter whether the M\'etivier's condition will be satisfied or not.

Let's return to our consideration in the beginning. It is worth pointing out that, for the classical semilinear elliptic Dirichlet problem \eqref{problem1-1} with free perturbation term, the classical Rellich–Kondrachov compact embedding theorem and the estimates of lower bound of Dirichlet eigenvalues play crucial roles in acquiring the existence and multiplicity of weak solutions (cf. \cite{Rabinowitz1982,Rabinowitz1986,Struwe2000}). Nevertheless, in the general finitely degenerate case and without the M\'etivier's condition, the lack of degenerate Rellich–Kondrachov compact embedding results and precise estimates of lower bound of Dirichlet eigenvalues will cause many new difficulties over a long period in the past. To our best knowledge, there is little information in literature about the existence and multiplicity of weak solutions for the semilinear subelliptic Dirichlet problem \eqref{problem1-1}.\par

In this aspect, Yung \cite{Yung2015} in 2015 established the sharp subelliptic Sobolev embedding of the weighted Sobolev space. Recently, the precise lower bound estimates of Dirichlet eigenvalues for the general H\"{o}rmander type subelliptic operator $\triangle_{X}$ have been obtained by Chen-Chen \cite{chen-chen2019}. With the help of these new results, we can now investigate the existence and multiplicity of weak solutions for the semilinear subelliptic Dirichlet problem \eqref{problem1-1} via the minimax method.

\par
To study the subelliptic equation \eqref{problem1-1} generated by vector fields $X$, we need to introduce the following corresponding weighted Sobolev spaces, namely
   \[ H_{X}^{1}(W)=\{u\in L^{2}(W)~|~X_{j}u\in L^{2}(W), j=1,\cdots,m\}. \]
 $H_{X}^{1}(W)$ is a Hilbert space endowed with norm $\|u\|^2_{H^{1}_{X}(W)}=\|u\|_{L^2(W)}^2+\|Xu\|_{L^2(W)}^2$, where  $\|Xu\|_{L^2(W)}^2=\sum_{j=1}^{m}\|X_{j}u\|_{L^2(W)}^2$. Then, we denote by $H_{X,0}^{1}(\Omega)$ the  closure of $C_{0}^{\infty}(\Omega)$ in $H_{X}^{1}(W)$, which is also a Hilbert space. \par
Now, we present our main results. First, if the perturbation term  in problem \eqref{problem1-1} vanishes, i.e $g\equiv 0$, we have
\begin{theorem}
 \label{thm1-1}
 Let $X=(X_{1},X_{2},\cdots,X_{m})$ be $C^{\infty}$ real vector fields defined in an open domain
  $W\subset\mathbb{R}^n$, which satisfy the H\"{o}rmander's condition (H) on $W$.  Assume that $\Omega\subset\subset W$ is a bounded connected open subset, and $\partial\Omega$ is smooth and non-characteristic for $X$. Then, for $g\equiv0$, we have following results:
  \begin{enumerate}
    \item [(1)] If  $f$ satisfies assumptions $(f_{1})$-$(f_{3})$, then problem \eqref{problem1-1} possesses a nontrivial weak solution $u\in H_{X,0}^1(\Omega)$.
    \item [(2)] If  $f$ satisfies assumptions $(f_2)$-$(f_4)$, then problem \eqref{problem1-1} admits an unbounded sequence of weak solutions $\{u_k\}_{k=1}^{\infty}$ in $H_{X,0}^1(\Omega)$.

  \end{enumerate}
 \end{theorem}

\begin{remark}
In fact, the assumptions $(f_1)$ and $(f_3)$ are sufficient but not necessary conditions for the existence of nontrivial weak solution. For example, let $X_{1}=\partial_{x_{1}}+2x_{2}\partial_{x_{3}}$ and $X_{2}=\partial_{x_{2}}-2x_{1}\partial_{x_{3}}$ be the vector fields on $\mathbb{R}^{3}$ which satisfy the H\"{o}rmander's condition and Metivier's condition with M\'{e}tivier index $\nu=4$ (see Definition \ref{def2-1} below). Then, considering $f(x,u)=\lambda u+u^{2}$ and $g(x)\equiv 0$ with $0<\lambda<\lambda_{1}$ (where $\lambda_{1}$ is the first Dirichlet eigenvalue for $-\triangle_{X}$ on $\Omega$), we can verify that $f(x,u)$ only satisfies assumption $(f_2)$, but in this case, one can also prove that the \eqref{problem1-1} possesses a nontrivial weak solution in $H_{X,0}^{1}(\Omega)$ (see Theorem 4.1 in \cite{Xu1994}).
\end{remark}

Then, for any non-vanishing free perturbation, we obtain

\begin{theorem}
 \label{thm1-2}
Let $X=(X_{1},X_{2},\cdots,X_{m})$ and $\Omega$ satisfy the same conditions as Theorem \ref{thm1-1}. Suppose that $f$ satisfies assumptions $(f_2)$-$(f_4)$ and the condition (A): $\frac{2p}{\tilde{\nu}(p-2)}-1>\frac{q}{q-1}$. Then for any function $g\in L^2(\Omega)$ such that $g\not\equiv0$, the problem $\eqref{problem1-1}$ has an unbounded sequence of weak solutions $\{u_k\}_{k=1}^{\infty}$ in $H_{X,0}^1(\Omega)$.
\end{theorem}
\begin{remark}
The condition (A): $\frac{2p}{\tilde{\nu}(p-2)}-1>\frac{q}{q-1}$ with $2<p<\frac{2\tilde{\nu}}{\tilde{\nu}-2}$ and $q>2$ in Theorem \ref{thm1-2} will be equivalent to $\left(2+\frac{1}{q-1}\right)\left(\frac{1}{2}-\frac{1}{p}\right)<\frac{1}{\tilde{\nu}}$. Particularly, we choose $q=p>2$, then we have $\left(2+\frac{1}{p-1}\right)\left(\frac{1}{2}-\frac{1}{p}\right)<\frac{1}{\tilde{\nu}}$, which implies that $(2\tilde{\nu}-2)p^{2}-(5\tilde{\nu}-2)p+2\tilde{\nu}<0$, i.e. $p$ satisfies $\frac{(5\tilde{\nu}-2)-\sqrt{9\tilde{\nu}^{2}-4\tilde{\nu}+4}}{4(\tilde{\nu}-1)}<p<\frac{(5\tilde{\nu}-2)
+\sqrt{9\tilde{\nu}^{2}-4\tilde{\nu}+4}}{4(\tilde{\nu}-1)}$. For $\tilde{\nu}\geq 3$, we can deduce that
$$
\frac{(5\tilde{\nu}-2)-\sqrt{9\tilde{\nu}^{2}-4\tilde{\nu}+4}}{4(\tilde{\nu}-1)}<2< \frac{(5\tilde{\nu}-2)+\sqrt{9\tilde{\nu}^{2}-4\tilde{\nu}+4}}{4(\tilde{\nu}-1)}
<\frac{2\tilde{\nu}}{\tilde{\nu}-2}.
$$
That means if we choose $2<q=p<\frac{(5\tilde{\nu}-2)+\sqrt{9\tilde{\nu}^{2}-4\tilde{\nu}+4}}{4(\tilde{\nu}-1)}$ for $\tilde{\nu}\geq 3$, then the condition (A) in Theorem \ref{thm1-2} will be satisfied. For example, for $n=2$, $X=(\partial_{x_1},x_1\partial_{x_2})$, then the M\'etivier's condition will be not satisfied and we can deduce that $\tilde{\nu}=3$, and $\frac{2\tilde{\nu}}{\tilde{\nu}-2}=6$. Then $\frac{(5\tilde{\nu}-2)+\sqrt{9\tilde{\nu}^{2}-4\tilde{\nu}+4}}{4(\tilde{\nu}-1)}=\frac{13+\sqrt{73}}{8}\geq 2.693$. If we choose $2<q=p<\frac{13+\sqrt{73}}{8}$, and $f(x,u)=|u|^{p-2}u$, then all conditions in Theorem \ref{thm1-1} and Theorem \ref{thm1-2} will be satisfied.
\end{remark}

\begin{remark}
If $\triangle_{X}=\triangle$, then the generalized M\'{e}tivier index $\tilde{\nu}=n$ and Theorem \ref{thm1-2} recovers to the Rabinowitz's classical results in \cite{Rabinowitz1982}.
\end{remark}
\begin{remark}
For degenerate case, if the H\"{o}rmander's condition is invalid, Morimoto and Xu \cite{Xu03} introduced a class of infinitely degenerate elliptic operators with logarithmic regularity estimate and studied the corresponding semilinear Dirichlet problems. They gained the existence and regularity of weak solution of the semilinear Dirichlet problems for infinitely degenerate elliptic operator without free perturbation terms. In recent years, their results have been generalized by Chen and Li in \cite{Li2009}. Moreover, Chen, Luo and Tian investigated the infinitely degenerate semilinear elliptic equations with a free perturbation term in \cite{Luo2015}.
\end{remark}

The rest of our paper is organized as follows. In Section \ref{Section2}, we introduce some preliminaries including the results for the finitely degenerate vector fields $X$ and minimax method in critical point theory. In Section \ref{Section3}, by using the degenerate Rellich–Kondrachov compact embedding theorem and the mountain pass theorem, we prove the existence and multiplicity of weak solutions for the problem \eqref{problem1-1} without the perturbation term. Finally, we study the problem \eqref{problem1-1} with non-vanishing  perturbation term in Section \ref{Section4}. With the help of lower bound estimates of Dirichlet eigenvalues of $-\triangle_{X}$, degenerate Rellich–Kondrachov compact embedding theorem and minimax method, we obtain that the problem \eqref{problem1-1} has infinitely many nontrivial unbounded weak solutions.

\section{Preliminaries}
\label{Section2}
 In this section, we introduce some results related to the finitely degenerate vector fields $X$ and minimax method in critical point theory.

 \subsection{Some results for the vector fields under the H\"{o}rmander's condition}
If the vector fields $X=(X_{1},X_{2},\cdots, X_{m})$ satisfy the H\"{o}rmander's condition (H) in an open domain $W$ in $\mathbb{R}^n$ and $\Omega\subset\subset W$, then there exists a smallest positive integer $Q\geq 1$
   such that the vector fields $X_{1},X_{2},\ldots,X_{m}$ together with their commutators
   of length at most $Q$ span the tangent space $T_{x}(W)$ at each point  $x\in \overline{\Omega}$. The integer $Q$ is called the  H\"{o}rmander's index of $\overline{\Omega}$ with respect to $X$.

    For the vector fields $X$ under the H\"{o}rmander's condition (H), M\'{e}tivier \cite{Metivier1976} gave the following stronger assumption in 1976.

 \begin{definition}[M\'{e}tivier's condition (cf. \cite{Metivier1976})]
    \label{def2-1}
   Let $\Omega\subset\subset W$ be a bounded connected open subset. Suppose that the vector fields $X$ satisfy H\"{o}rmander's condition (H) on an open domain $W$ in $\mathbb{R}^n$ and have H\"{o}rmander's index $Q$ on $\overline{\Omega}$. For each $x\in \overline{\Omega}$, let $V_{j}(x)~ (1\leq j\leq Q)$
be the subspaces of the tangent space at $x$  spanned by all
commutators of $X_{1},\ldots,X_{m}$ with length at most $j$. If the dimension of $V_{j}(x)$ is constant $\nu_{j}$ in a neighborhood of each $x\in \overline{\Omega}$, then we say the vector fields $X$ satisfy the M\'{e}tivier's condition on $\Omega$. Moreover, the M\'{e}tivier index is defined as
\begin{equation}\label{2-1}
  \nu=\sum_{j=1}^{Q}j(\nu_{j}-\nu_{j-1}),\qquad \nu_{0}:=0,
\end{equation}
which is also called the Hausdorff dimension of $\Omega$ related to the subelliptic metric induced by the vector fields $X$.
\end{definition}

It is worth mentioning that the M\'{e}tivier's condition is not only an important assumption on study of finitely degenerate elliptic operator, but also a crucial condition in sub-Riemannian geometry. The M\'etivier's condition possesses
a strong restriction on the vector fields $X$ satisfying
H\"ormander's condition, in which the Lie algebra generated by
the vector fields $X_{1},X_{2},\ldots,X_{m}$ takes a constant
structure and the vector fields can be locally approximated by some
homogeneous left invariant vector fields defined on the
corresponding Carnot group (cf. \cite{stein1976}). However, there exist many vector fields satisfying the H\"ormander's condition but without the restriction of the M\'{e}tivier's condition (e.g. the Grushin vector fields $X_{1}=\partial_{x_{1}}, X_{2}=x_{1}\partial_{x_{2}}$). In this general case, we need to introduce the following generalized M\'{e}tivier's index, which is also called the non-isotropic dimension of $\Omega$ related to $X$ by \cite{Yung2015}.

\begin{definition}[Generalized M\'{e}tivier's index (cf. \cite{chen-chen2019,chenluo})]
\label{def2-2}
 With the same
notations as Definition \ref{def2-1}, we denote by $\nu_{j}(x)$ the dimension of vector space $V_{j}(x)$ at point $x\in \overline{\Omega}$. The pointwise homogeneous dimension at $x$ is given
by
\begin{equation}\label{2-2}
  \nu(x):=\sum_{j=1}^{Q}j(\nu_{j}(x)-\nu_{j-1}(x)),\qquad \nu_{0}(x):=0.
\end{equation}
Then we define
\begin{equation}\label{2-3}
  \tilde{\nu}:=\max_{x\in\overline{\Omega}} \nu(x)
\end{equation}
as the generalized M\'{e}tivier index of $\Omega$. Observe that \eqref{2-2} implies  $n+Q-1\leq \tilde{\nu}< nQ$ for $Q>1$, and  $\tilde{\nu}=\nu$ if the M\'{e}tivier's condition is satisfied.
\end{definition}

Now, we introduce the following weighted Sobolev embedding theorem and weighted Poincar\'{e} inequality related to vector fields $X$.
\begin{proposition}[Weighted Sobolev Embedding Theorem]
 \label{pro2-1}
 Let $X$ and $\Omega$ satisfy the same assumptions as Theorem \ref{thm1-1}. Denote by $\tilde{\nu}$ the generalized M\'{e}tivier index of $X$ on $\Omega$.
   Then for $1\leq p<\tilde{\nu}$,  there exists a constant $C=C(\Omega,X)>0$, such that for all
    $u\in C^{\infty}(\overline{\Omega})$, the inequality
\begin{equation}\label{2-4}
\|u\|_{L^{q}(\Omega)}\leq C\left(\|Xu\|_{L^{p}(\Omega)}+\|u\|_{L^{p}(\Omega)}\right)
\end{equation}
holds for $q=\frac{\tilde{\nu}p}{\tilde{\nu}-p}$.
\end{proposition}
\begin{proof}
See Corollary 1 in \cite{Yung2015}.
\end{proof}
\begin{corollary}
\label{corollary2-1}
In particular, if $\tilde{\nu}\geq n+Q-1\geq 3$. Putting $p=2$ into  Proposition \ref{pro2-1}, we can deduce that the embedding
\begin{equation}
\label{2-5}
H_{X,0}^{1}(\Omega)\hookrightarrow L^{q}(\Omega)
\end{equation}
is bounded for $1\leq q\leq \frac{2\tilde{\nu}}{\tilde{\nu}-2}:=2_{\tilde{\nu}}^{*}$.
\end{corollary}

\begin{proposition}[Weighted Poincar\'{e} Inequality]
\label{Poincare}
 Suppose that $X$ and $\Omega$ satisfy the same assumptions as Theorem \ref{thm1-1}. Then the first Dirichlet eigenvalue $\lambda_{1}$ of self-adjoint operator $-\triangle_{X}$ is positive. Moreover, we have the following weighted Poincar\'{e} inequality
\begin{equation}\label{2-6}
  \lambda_{1}\int_{\Omega}{|u|^2dx}\leq \int_{\Omega}|Xu|^2dx,~~ \forall u\in H_{X,0}^{1}(\Omega).
  \end{equation}
\end{proposition}
\begin{proof}
See Proposition 2.2 in \cite{chen-chen2019} and Lemma 3.2 in \cite{Jost1998}.

\end{proof}

On the other hand, we have the following subelliptic estimates.

\begin{proposition}[Subelliptic Estimates]
\label{subelliptic-estimate}
Let the vector fields $X$  satisfy the H\"{o}rmander condition (H) in an open domain $W$ in $\mathbb{R}^n$. Then, for any open subset $\Omega\subset\subset W$, there exist constant $\epsilon_{0}>0$ and  $C>0$ such that
\begin{equation}\label{2-7}
  \|u\|_{H^{\epsilon_{0}}(\mathbb{R}^n)}^2\leq C\left(\|Xu\|_{L^2(\mathbb{R}^n)}^{2}+\|u\|_{L^2(\mathbb{R}^n)}^2\right), ~~\mbox{for all } u\in H_{X,0}^{1}(\Omega),
\end{equation}
where $\|u\|_{H^{s}(\mathbb{R}^n)}^2=\int_{\mathbb{R}^n}(1+|\xi|^{2})^{s}|\hat{u}(\xi)|^{2}d\xi$.
\end{proposition}
\begin{proof}
 See Theorem 17 in \cite{stein1976}.
\end{proof}
\begin{corollary}
\label{corollary2-2}
According to classical Sobolev embedding theory (cf. Proposition 4.4 in \cite{Taylor2011}), we know that the embedding $H^{\epsilon_{0}}(\Omega)\hookrightarrow L^2(\Omega)$ is compact. Therefore, the embedding $H_{X,0}^1(\Omega)\hookrightarrow L^2(\Omega)$ is compact.
\end{corollary}

With the help of Proposition \ref{Poincare} and Proposition \ref{subelliptic-estimate},  the following proposition can be derived.
\begin{proposition}
\label{prop2-4}
Assume that $X$ and $\Omega$ satisfy the same assumptions as Theorem \ref{thm1-1}.  Then the subelliptic Dirichlet problem
  \begin{equation}\label{2-8}
    \left\{
      \begin{array}{ll}
    -\triangle_{X}u=\lambda u   , & \hbox{$x\in \Omega$;} \\
        u=0, & \hbox{$x\in \partial\Omega$.}
      \end{array}
    \right.
  \end{equation}
has a sequence of discrete Dirichlet eigenvalues $0<\lambda_1\leq\lambda_2\leq\cdots\leq\lambda_{k-1}\leq\lambda_k\leq\cdots$,
and $\lambda_{k}\to +\infty $ as $k\to +\infty$. Moreover, the corresponding eigenfunctions $\{\varphi_{k}\}_{k=1}^{\infty}$ constitute an orthonormal basis of $L^2(\Omega)$ and also an orthogonal basis of $H_{X,0}^{1}(\Omega)$.
\end{proposition}

In 2015, Chen and Luo \cite{chenluo} studied the subelliptic Dirichlet eigenvalue problem \eqref{2-8} and obtained the sharp lower bound $\lambda_{k}\geq C_{1}\cdot k^{\frac{2}{\tilde{\nu}}}$ for Grushin type vector fields. Then, their results have been generalized in \cite{CCD,chen-chen2019-jpsdo,chen-chen2019-ata,Chen-Zhou2017}. Recently in \cite{chen-chen2019}, the authors have proved that this sharp lower bound is also valid for general H\"{o}rmander type subelliptic operator $-\triangle_{X}$, namely
\begin{proposition}
\label{prop2-5}
Let $X$ and $\Omega$ satisfy the same assumptions as Theorem \ref{thm1-1}. Denoting by $\lambda_{k}$ the $k$-th Dirichlet eigenvalue of $-\triangle_{X}$ on $\Omega$, then
\begin{equation}\label{2-9}
\lambda_{k}\geq C_{1}\cdot k^{\frac{2}{\tilde{\nu}}},\qquad \forall ~k\geq 1,
\end{equation}
where $C_{1}>0$ is a constant depending on $X$ and $\Omega$, and $\tilde{\nu}$ is the generalized M\'{e}tivier index of $X$ on $\Omega$.
\end{proposition}
\begin{proof}
See Theorem 1.2  in \cite{chen-chen2019}.
\end{proof}
\begin{remark}
In \cite{chen-chen2019-jmpa}, the authors proved that, if $-\triangle_{X}$ is a H\"{o}rmander type subelliptic operator on compact manifold, then the eigenvalues also admit the lower bound \eqref{2-9}.
\end{remark}

Next, we present the following abstract version of the Rellich–Kondrachov compactness theorem, which will be used in developing the compact embedding result for weighted Sobolev space.
\begin{proposition}
\label{prop2-6}
Let $Y$ be a set equipped with a finite measure $\mu$. Assume that a linear normed space $G$ of measurable functions on $Y$ has the following two properties:
\begin{enumerate}
  \item [(1)] There exists a constant $q>1$ such that the embedding $G\hookrightarrow L^q(Y,\mu)$ is bounded;
  \item [(2)] Every bounded sequence in $G$ contains a subsequence that converges almost everywhere.
\end{enumerate}
 Then the embedding $G\hookrightarrow L^s(Y,\mu)$ is compact for every $1\leq s<q$.
\end{proposition}
\begin{proof}
See Theorem 4 in \cite{Hajlasz1998}.
\end{proof}

By Proposition \ref{prop2-6}, we establish the following  degenerate
compact embedding theorem of $H_{X,0}^{1}(\Omega)$.
\begin{proposition}[Degenerate Rellich–Kondrachov
Compact Embedding Theorem]
\label{prop2-7}
Let $X$ and $\Omega$ satisfy the same assumptions as Theorem \ref{thm1-1}. Then the embedding
\[ H_{X,0}^1(\Omega)\hookrightarrow L^s(\Omega) \]
is compact for every $s\in [1,2_{\tilde{\nu}}^*)$, where $2_{\tilde{\nu}}^*=\frac{2\tilde{\nu}}{\tilde{\nu}-2}$ and $\tilde{\nu}\geq 3$ is the generalized M\'{e}tivier index defined in \eqref{2-3}.
\end{proposition}
\begin{proof}
From \eqref{2-5}, we know the embedding $H_{X,0}^1(\Omega)\hookrightarrow L^{2_{\tilde{\nu}}^*}(\Omega)$ is bounded. For any bounded sequence $\{u_k\}_{k=1}^{\infty}$ in $H_{X,0}^1(\Omega)$, there exists a subsequence $\{u_{k_{j}}\}_{j=1}^{\infty}\subset\{u_k\}_{k=1}^{\infty}$ such that $u_{k_{j}}\rightharpoonup u$ weakly in $H_{X,0}^1(\Omega)$. By Corollary \ref{corollary2-2},  $H_{X,0}^1(\Omega)\hookrightarrow L^2(\Omega)$ is compact, which gives $u_{k_{j}}\to u$ in $L^2(\Omega)$. Therefore, the Riesz theorem allows us to find a subsequence $\{v_{j}\}_{j=1}^{\infty}\subset  \{u_{k_{j}}\}_{j=1}^{\infty}$  such that $v_{j}\to u$ almost everywhere on $\Omega$ as $j\to +\infty$. Hence, we conclude from Proposition \ref{prop2-6} that the embedding $H_{X,0}^1(\Omega)\hookrightarrow L^s(\Omega)$ is compact for $s\in [1,2_{\tilde{\nu}}^*)$.
 \end{proof}

\subsection{Minimax method in critical point theory}
In this part, we introduce some propositions in minimax method.
\begin{proposition}
\label{prop2-8}
Let $\Omega\subset \mathbb{R}^{n}$ be a bounded domain and suppose that $g$ satisfies
\begin{enumerate}
  \item [(1)] $g\in C(\overline{\Omega}\times \mathbb{R})$.
  \item [(2)] There are constants $C>0$ and $s\geq1$ such that
  \[|g(x,u)|\leq C(1+{|u|}^s)\]
  for all $x\in \overline{\Omega}, u\in \mathbb{R}$.
\end{enumerate}
 Then the map $u(x)\mapsto g(x,u(x))$ belongs to $ C(L^{sp}(\Omega), L^{p}(\Omega))$ for $1\leq p<\infty$.
\end{proposition}
\begin{proof}
See Proposition B.1 in \cite{Rabinowitz1986}.

\end{proof}

\begin{proposition}
\label{prop2-9}
Let $X=(X_1,X_{2},\cdots,X_m)$ and $\Omega$ satisfy the same conditions as  Theorem $\ref{thm1-1}$. Suppose
$ f\in C(\overline{\Omega}\times \mathbb{R})$  and $F(x,u)=\int_{0}^{u}f(x,v)dv$. Moreover, there exists a constant $C>0$ such that the following growth conditions are satisfied:

\begin{enumerate}
  \item [(1)] There exists $1\leq s_1\leq \frac{2\tilde{\nu}}{\tilde{\nu}-2}$ such that
\[|F(x,u)|\leq C(1+{|u|}^{s_1})~~~\mbox{for any}~~x\in \overline{\Omega}.\]
  \item [(2)] There exists $1\leq s_2\leq \frac{\tilde{\nu}+2}{\tilde{\nu}-2}$ such that
\[ |f(x,u)|\leq C(1+{|u|}^{s_2})~~~\mbox{for any}~~x\in \overline{\Omega}. \]
\end{enumerate}
 Then $J(u)=\int_{\Omega}F(x,u(x))dx$ defines a $C^1$-functional on $H_{X,0}^1(\Omega)$, that is to say, the Fr\'{e}chet derivative of $J$ exists and is continuous on $H_{X,0}^1(\Omega)$. Furthermore, the Fr\'{e}chet derivative (denoted by $DJ$) is given by
 \[\langle DJ(u),v\rangle=\int_{\Omega}f(x,u)vdx,~~\forall~ v\in H_{X,0}^1(\Omega),\]
   where $\langle\cdot , \cdot\rangle$ denotes the pairing between $H_{X}^{-1}(\Omega)$ and $H_{X,0}^1(\Omega)$.
  \end{proposition}

 \begin{proof}
First, from the first growth condition $|F(x,u)|\leq C(1+{|u|}^{s_1})$ with $1\leq s_1\leq \frac{2\tilde{\nu}}{\tilde{\nu}-2}$ and \eqref{2-5}, we
can deduce that the functional $J(u)=\int_{\Omega}F(x,u(x))dx$ on $H_{X,0}^1(\Omega)$ is well-defined.

The second growth condition indicates
\begin{equation}\label{2-10}
|f(x,u)| \leq C(1+{|u|}^{s_2})\leq \tilde{C}(1+{|u|}^{\frac{\tilde{\nu}+2}{\tilde{\nu}-2}}).
 \end{equation}
Combining \eqref{2-10} and \eqref{2-5}, we obtain that,  for any $u,v\in H_{X,0}^{1}(\Omega)$,
\begin{equation}\label{2-11}
\begin{aligned}
\left|\int_{\Omega}f(x,u)vdx \right|&\leq \tilde{C}\int_{\Omega}\left(|v|+|u|^{\frac{\tilde{\nu}+2}{\tilde{\nu}-2}}|v|\right)dx\\
&\leq C_{1}\|v\|_{H_{X,0}^{1}(\Omega)}+\tilde{C}\int_{\Omega}|u|^{\frac{\tilde{\nu}+2}{\tilde{\nu}-2}}|v|dx\\
&\leq C_{1}\|v\|_{H_{X,0}^{1}(\Omega)}+\tilde{C}\left(\int_{\Omega}\left( |u|^{\frac{\tilde{\nu}+2}{\tilde{\nu}-2}} \right)^{\frac{2\tilde{\nu}}{\tilde{\nu}+2}} dx\right)^{\frac{\tilde{\nu}+2}{2\tilde{\nu}}}\cdot \left(\int_{\Omega
}|v|^{\frac{2\tilde{\nu}}{\tilde{\nu}-2}}dx\right)^{\frac{\tilde{\nu}-2}{2\tilde{\nu}}}\\
&= C_{1}\|v\|_{H_{X,0}^{1}(\Omega)}+\tilde{C}\left(\int_{\Omega}|u|^{\frac{2\tilde{\nu}}{\tilde{\nu}-2}}dx  \right)^{\frac{\tilde{\nu}+2}{2\tilde{\nu}}}\cdot \|v\|_{L^{\frac{2\tilde{\nu}}{\tilde{\nu}-2}}(\Omega)}\\
&\leq \left(C_{1}+C_{2}\|u\|_{H_{X,0}^{1}(\Omega)}^{\frac{\tilde{\nu}+2}{\tilde{\nu}-2}}\right)\|v\|_{H_{X,0}^{1}(\Omega)},
\end{aligned}
\end{equation}
where $C_{1}, C_{2}$ are positive constants. Therefore, \eqref{2-11} implies the G\^{a}teaux derivative of the functional $J$ at $u\in H_{X,0}^{1}(\Omega)$ given by
 \begin{equation}\label{2-12}
  \langle J'(u),v\rangle=\int_{\Omega}f(x,u)vdx\quad \forall v\in H_{X,0}^1(\Omega)
 \end{equation}
is well-defined and belongs to $H_{X}^{-1}(\Omega)$.

On the other hand, for $u,u_0 \in H_{X,0}^1(\Omega)$, we have
\begin{equation}\label{2-13}
 \begin{aligned}
 &\|J'(u)-J'(u_0)\|_{{H_X^{-1}}(\Omega)}\\
&=\sup\limits_{
 {v\in H_{X,0}^1(\Omega)},{\|v\|_{{H_{X,0}^1}(\Omega)}\leq1}}|\langle J'(u)-J'(u_0),v\rangle|   \\
 &\leq\sup\limits_{{v\in H_{X,0}^1(\Omega)},
{\|v\|_{{H_{X,0}^1}(\Omega)}\leq1}}\int_{\Omega}|f(x,u)-f(x,u_0)||v|dx\\
 &\leq\sup\limits_{
 {v\in H_{X,0}^1(\Omega)}
,{\|v\|_{{H_{X,0}^1}(\Omega)}\leq1}}\left(\int_{\Omega}|f(x,u)-f(x,u_0)|^{\frac{2\tilde{\nu}}{\tilde{\nu}+2}}dx
\right)^{\frac{\tilde{\nu}+2}{2\tilde{\nu}}}\|v\|_{L^{\frac{2\tilde{\nu}}{\tilde{\nu}-2}}(\Omega)}\\
 &\leq C\left(\int_{\Omega}|f(x,u)-f(x,u_0)|^{\frac{2\tilde{\nu}}{\tilde{\nu}+2}}dx\right)^{\frac{\tilde{\nu}+2}{2\tilde{\nu}}}.
 \end{aligned}
 \end{equation}
If
$u\rightarrow u_0$ in $H_{X,0}^{1}(\Omega)$, by \eqref{2-5}  we have $u\rightarrow u_0 $ in $L^{\frac{2\tilde{\nu}}{\tilde{\nu}-2}}(\Omega)$. Take $s=\frac{\tilde{\nu}+2}{\tilde{\nu}-2},~p=\frac{2\tilde{\nu}}{\tilde{\nu}+2}$ in Proposition \ref{prop2-8}. Then, the Proposition \ref{prop2-8} and \eqref{2-10} tell us $u\mapsto f(\cdot,u(\cdot))$ is continuous from $L^{{\frac{2\tilde{\nu}}{\tilde{\nu}-2}}}(\Omega)$ into $L^{{\frac{2\tilde{\nu}}{\tilde{\nu}+2}}}(\Omega)$. Hence, it follows from \eqref{2-13} that $J'(u)\rightarrow J'(u_0)~\mbox{in}~H_X^{-1}(\Omega)$. This means $J$ has continuous G\^{a}teaux derivative on $H_{X,0}^{1}(\Omega)$. Furthermore, by Proposition 1.3 in \cite{Willem1997}, we obtain $J\in C^{1}(H_{X,0}^{1}(\Omega),\mathbb{R})$ and the Fr\'{e}chet derivative $DJ(u)=J'(u)$ for all $u\in H_{X,0}^{1}(\Omega)$.

 \end{proof}

Suppose $V$ is a real Banach space and $V^{*}$ is the dual space of $V$. Then, we state some basic definitions and propositions in minimax method.
 \begin{definition}
 \label{def2-3}
For $E\in C^{1}(V, \mathbb{R})$, we say a sequence $\{u_m\}$ in $V$ is a Palais-Smale (PS) sequence for functional $E$ if ${|E(u_m)|}\leq C$ uniformly in m, and ${\|DE(u_m)\|}_{V^*}\rightarrow0$ as $m\rightarrow\infty$.
 \end{definition}

  \begin{definition}
  \label{def2-4}
  A functional $E\in C^{1}(V, \mathbb{R})$ satisfies Palais-Smale condition (henceforth denoted by (PS) condition) if any Palais-Smale sequence has a subsequence which is convergent in $V$.
  \end{definition}

 In fact, it is easier to verify (PS) condition for functional $E$ by following proposition.

\begin{proposition}
\label{prop2-10}

 Suppose that $V$ is a reflexive real Banach space, and the functional $E\in C^{1}(V, \mathbb{R})$ admits the following properties:
\begin{enumerate}
  \item [(1)] Any (PS) sequence for $E$ is bounded in $V$.
  \item [(2)] For any $u\in V$ we can decompose
 \[DE(u)=L(u)+K(u),\]
 where $L:V\rightarrow V^*$ is a fixed bounded invertible linear map and the operator $K$ maps bounded sets in $V$ to relatively compact sets in $V^*$.
\end{enumerate}
Then $E$ satisfies (PS) condition.
\end{proposition}
\begin{proof}
See Chapter 2, Proposition 2.2 in \cite{Struwe2000}.
\end{proof}

\begin{proposition}[Mountain Pass Theorem]
\label{prop2-11}
 Assume that $E\in C^1(V, \mathbb{R})$ satisfies (PS) condition and
\begin{enumerate}
  \item [(1)] $E(0)=0$,
  \item [(2)]$\exists\rho>0,\alpha>0$ such that for ${\|u\|_{V}}=\rho$,  $ E(u)\geq\alpha$,
  \item [(3)] $\exists u_1\in V$ such that ${\|u_1\|_{V}}>\rho$ and $E(u_1)\leq 0$.
\end{enumerate}
Then $E$ possesses a critical value $c\geq \alpha$, which can be characterized as
 \[c=\inf\limits_{l\in P}\sup\limits_{u\in l([0,1])}E(u)\]
where
 \[P=\{l\in C([0,1],V)|~ l(0)=0,l(1)=u_1\}.\]

\end{proposition}

\begin{proof}
See Theorem 2.2 in \cite{Rabinowitz1986}.
\end{proof}

\begin{proposition}[Symmetrical Mountain Pass
Theorem]
\label{prop2-12}
 Let $V$ be an infinite dimensional Banach space and suppose $E\in C^1(V, \mathbb{R})$ satisfies (PS) condition, $E(u)=E(-u)$ for all $u\in V$, and $E(0)=0$. If $V=V^-\oplus V^+$, where $V^-$ is finite dimensional, and $E$ satisfies the following conditions:
\begin{enumerate}
  \item [(1)] $\exists\alpha>0, \rho>0$, such that for any $u\in V^+$ with ${\|u\|_{V}}=\rho$, $E(u)\geq\alpha$.
  \item [(2)] For any finite dimensional subspace $W\subset V$ there is a constant $R=R(W)>0$ such that $E(u)\leq0$ for $u \in W, \|u\|_{V}\geq R$.
\end{enumerate}
Then $E$ possesses an unbounded sequence of critical values.
\end{proposition}

\begin{proof}
See Theorem 9.12 in \cite{Rabinowitz1986}.
\end{proof}

\section{Existence and multiplicity of weak solutions without perturbation term}
\label{Section3}
In this section, we focus on the existence and multiplicity of weak solutions for problem \eqref{problem1-1} without the perturbation term, namely
\begin{equation}\label{problem3-1}
\left\{
      \begin{array}{cc}
      -\triangle_{X}u=f(x,u) & \mbox{in}~\Omega, \\[2mm]
      u=0           & \mbox{on}~\partial\Omega.
      \end{array}
 \right.
\end{equation}

\begin{definition}
We say that $u\in H_{X,0}^{1}(\Omega)$ is a weak solution of \eqref{problem3-1} if
\begin{equation}
\label{3-2}
  \int_{\Omega}Xu\cdot Xvdx-\int_{\Omega}f(x,u)vdx=0,~~~\mbox{for all}~v\in H_{X,0}^{1}(\Omega).
\end{equation}
\end{definition}
 Now, we introduce the following energy functional $E: H_{X,0}^{1}(\Omega)\to \mathbb{R}$, defined as
\begin{equation}\label{3-3}
  E(u)=\frac{1}{2}\int_{\Omega}{{|{Xu}|}^2} dx-\int_{\Omega}F(x,u) dx.
\end{equation}
Then we have
\begin{proposition}
\label{prop3-1}
If $f(x,u)$ satisfies the assumption $(f_{2})$, then
  \[E(u)=\frac{1}{2}\int_{\Omega}{{|{Xu}|}^2} dx-\int_{\Omega}F(x,u) dx\]
belongs to $C^{1}(H_{X,0}^{1}(\Omega), \mathbb{R})$. Thus the semilinear equation \eqref{problem3-1} is the Euler-Lagrange equation of the variational problem for the energy functional \eqref{3-3}. Moreover, the Fr\'{e}chet derivative of $E$ at $u$ is given by
\begin{equation}\label{3-4}
  \langle DE(u),v\rangle=\int_{\Omega}Xu\cdot Xvdx-\int_{\Omega}f(x,u)vdx,~~\mbox{for all}~~v\in H_{X,0}^{1}(\Omega).
\end{equation}
Therefore, the critical point of $E$ in $H_{X,0}^{1}(\Omega)$ is the weak solution of \eqref{problem3-1}.
\end{proposition}

\begin{proof}
Let
 \[ E(u)=\frac{1}{2}\int_{\Omega}|Xu|^2dx-\int_{\Omega}F(x,u)dx=:I(u)-J(u),~~\mbox{for}~u\in H_{X,0}^1(\Omega),\]
where
\[ I(u)=\frac{1}{2}\int_{\Omega}|Xu|^2dx\quad \mbox{and}\quad J(u)=\int_{\Omega}F(x,u)dx.\]
Clearly, the functional $I(u)$ is well-defined on $H_{X,0}^{1}(\Omega)$ and has the G\^{a}teaux derivative $I'(u)$ given by
\begin{equation}\label{3-5}
\langle I'(u),v\rangle=\int_{\Omega}Xu\cdot Xvdx, \quad\forall v\in H_{X,0}^1(\Omega).
\end{equation}
For $u,u_0 \in H_{X,0}^1(\Omega)$, we have
\begin{align*}
 \|I'(u)-I'(u_0)\|_{{H_X^{-1}}(\Omega)}&=\sup\limits_{
 {v\in H_{X,0}^1(\Omega)},{\|v\|_{{H_{X,0}^1}(\Omega)}\leq1}}|\langle I'(u)-I'(u_0),v\rangle|   \\
 &\leq\sup\limits_{{v\in H_{X,0}^1(\Omega)},
{\|v\|_{{H_{X,0}^1}(\Omega)}\leq1}}\int_{\Omega}|Xu-Xu_0||Xv|dx\\
&\leq\sup\limits_{{v\in H_{X,0}^1(\Omega)},
{\|v\|_{{H_{X,0}^1}(\Omega)}\leq1}}\left(\int_{\Omega}|X(u-u_0)|^2dx\right)^{\frac{1}{2}}\|v\|_{{H_{X,0}^1}(\Omega)}\\
&\leq\|u-u_0\|_{{H_{X,0}^1}(\Omega)}.
 \end{align*}
 Hence, we obtain that  $I'(u)\rightarrow I'(u_0)$ in $H_X^{-1}(\Omega)$, provided $u\rightarrow u_0$ in $H_{X,0}^1(\Omega)$ .
 That means $I\in C^{1}(H_{X,0}^{1}(\Omega), \mathbb{R})$ and the Fr\'{e}chet derivative $DI(u)=I'(u)$ for all $u\in H_{X,0}^{1}(\Omega)$.\par
   Our task now is to prove that $J\in C^{1}(H_{X,0}^{1}(\Omega), \mathbb{R})$. By the assumption $(f_{2})$, there exist $2<p<\frac{2\tilde{\nu}}{\tilde{\nu}-2}$ and constant $C>0$ such that
\begin{equation}\label{3-6}
|f(x,v)|\leq C(1+|v|^{p-1}), \quad \mbox{for all}~x\in\overline{\Omega},~~ v\in \mathbb{R}.
\end{equation}
Integrating \eqref{3-6} with respect to $v$, we get
\begin{equation}\label{3-7}
|F(x,u)|=\left|\int_{0}^{u}f(x,v)dv\right|\leq   \left|\int_{0}^{u}|f(x,v)|dv\right|\leq C(|u|+|u|^{p})\leq \widetilde{C}(1+|{u}|^{s}),
\end{equation}
where $s$ is a positive constant satisfying  $2<p<s\leq{\frac{2\tilde{\nu}}{\tilde{\nu}-2}}$. Combining \eqref{3-6}, \eqref{3-7} and Proposition \ref{prop2-9}, we conclude $J\in C^{1}(H_{X,0}^{1}(\Omega), \mathbb{R})$ and the Fr\'{e}chet derivative
$DJ$ at $u$ is given by
 \begin{equation}\label{3-8}
   \langle DJ(u),v\rangle=\int_{\Omega}f(x,u)vdx, ~\forall v\in H_{X,0}^1(\Omega).
 \end{equation}

 Therefore, the functional $E\in C^{1}(H_{X,0}^{1}(\Omega), \mathbb{R})$ and its Fr\'{e}chet derivative is given by
\begin{equation}\label{3-9}
 \langle DE(u),v\rangle=\int_{\Omega}\left(Xu\cdot Xv-f(x,u)v\right)dx,~\forall v\in H_{X,0}^1(\Omega).
\end{equation}

\end{proof}

\begin{proposition}
\label{prop3-2}
If $f(x,u)$ satisfies the assumptions $(f_{2})$ and $(f_{3})$, then $C^{1}$-functional
  \[E(u)=\frac{1}{2}\int_{\Omega}{{|{Xu}|}^2} dx-\int_{\Omega}F(x,u) dx\]
satisfies (PS) condition.
\end{proposition}

\begin{proof}
 Consider the following bilinear form
\begin{equation}\label{3-10}
a[u,v]=\int_{\Omega}Xu\cdot Xvdx,~~\mbox{for}~u,v ~\mbox{in}~H_{X,0}^1(\Omega).
\end{equation}
Clearly,
\begin{equation}\label{3-11}
  |a[u,v]|=\left|\int_{\Omega}Xu\cdot Xvdx\right|\leq \|u\|_{H_{X,0}^{1}(\Omega)}\cdot \|v\|_{H_{X,0}^{1}(\Omega)}.
\end{equation}
Therefore, for any $u\in  H_{X,0}^{1}(\Omega)$, the bilinear form $a[u,\cdot]\in H_{X}^{-1}(\Omega)$ determines a functional $L(u)\in H_{X}^{-1}(\Omega)$ given by
\begin{equation}\label{3-12}
  \langle L(u),v\rangle=a[u,v]~~\mbox{for all}~v\in H_{X,0}^{1}(\Omega).
\end{equation}
Note that  \eqref{3-12} also implies that $L: u\mapsto L(u)$ is a linear operator from  $H_{X,0}^{1}(\Omega)$ to $H_{X}^{-1}(\Omega)$. By \eqref{3-11} and \eqref{3-12}, we have
\begin{equation}\label{3-13}
\|L(u)\|_{H_{X}^{-1}(\Omega)}=\sup_{v\in H_{X,0}^{1}(\Omega), \|v\|_{H_{X,0}^{1}(\Omega)}\leq 1}|\langle L(u),v\rangle|\leq \|u\|_{H_{X,0}^{1}(\Omega)},
\end{equation}
which implies that $L$ is a bounded linear operator from  $H_{X,0}^{1}(\Omega)$ to $H_{X}^{-1}(\Omega)$. Owing to the weighted Poincar\'{e} inequality (Proposition \ref{Poincare}), we have
\begin{equation}\label{3-14}
  a[u,u]=\int_{\Omega}|Xu|^{2}dx\geq \frac{\lambda_{1}}{1+\lambda_{1}}\|u\|_{H_{X,0}^{1}(\Omega)}^{2}~~\mbox{for all}~u\in H_{X,0}^{1}(\Omega).
\end{equation}
Combining \eqref{3-11}, \eqref{3-14} and Lax-Milgram theorem, we obtain that $L:{H_{X,0}^1}(\Omega)\mapsto {H_{X}^{-1}(\Omega)}$ is a bounded invertible linear map.\par

On the other hand, for any $u\in H_{X,0}^{1}(\Omega)$, we can deduce  from estimate \eqref{2-11} that the linear functional $K(u)$ on $H_{X,0}^{1}(\Omega)$ given by
\begin{equation}\label{3-16}
\langle K(u),v\rangle =-\int_{\Omega}f(x,u)vdx~~~\mbox{for}~~v\in H_{X,0}^{1}(\Omega)
\end{equation}
belongs to $H_{X}^{-1}(\Omega)$. We then show that the operator $K: u\mapsto K(u)$ from $H_{X,0}^1(\Omega)$ to $H_{X}^{-1}(\Omega)$ maps bounded sets in $H_{X,0}^{1}(\Omega)$ to relatively compact sets in $H_{X}^{-1}(\Omega)$. For any bounded sequence $\{u_k\}_{k=1}^{\infty}$ in $H_{X,0}^1(\Omega)$, there exists a subsequence $\{u_{k_j}\}_{j=1}^{\infty}\subset \{u_k\}_{k=1}^{\infty}$ such that $u_{k_j}\rightharpoonup u$ weakly in $H_{X,0}^1(\Omega)$ as $j\to +\infty$. Since $1<(p-1)\cdot \frac{2\tilde{\nu}}{\tilde{\nu}+2}<\frac{2\tilde{\nu}}{\tilde{\nu}-2}$, by Proposition \ref{prop2-7} we know that the embedding $H_{X,0}^{1}(\Omega)\hookrightarrow L^{(p-1)\cdot \frac{2\tilde{\nu}}{\tilde{\nu}+2}}(\Omega)$ is compact, which implies that $u_{k_{j}}\to u$ in $L^{(p-1)\cdot \frac{2\tilde{\nu}}{\tilde{\nu}+2}}(\Omega)$. Besides, the Proposition \ref{prop2-8} tells us the map $u\mapsto f(\cdot, u(\cdot))$ is continuous from $L^{(p-1)\cdot \frac{2\tilde{\nu}}{\tilde{\nu}+2}}(\Omega)$ into $L^{\frac{2\tilde{\nu}}{\tilde{\nu}+2}}(\Omega)$. Thus, we have $f(\cdot,u_{k_j})\rightarrow f(\cdot,u)$ in $L^{\frac{2\tilde{\nu}}{\tilde{\nu}+2}}(\Omega)$ as $j\to +\infty$.
Moreover, the H\"{o}lder's inequality and \eqref{2-5} yield that
\begin{equation}\label{3-17}
\begin{aligned}
&\|K(u_{k_j})-K(u)\|_{{H_X^{-1}}(\Omega)}\\
&=\sup\limits_{
 {v\in H_{X,0}^1(\Omega)},{\|v\|_{{H_{X,0}^1}(\Omega)}\leq1}}|\langle K(u_{k_j})-K(u),v\rangle|\\
&\leq \sup\limits_{
 {v\in H_{X,0}^1(\Omega)},{\|v\|_{{H_{X,0}^1}(\Omega)}\leq1}}\int_{\Omega}|f(x,u_{k_j})-f(x,u)||v|dx\\
&\leq \sup\limits_{
 {v \in H_{X,0}^1(\Omega)},{\|v\|_{{H_{X,0}^1}(\Omega)}\leq1}}
 \left(\int_{\Omega}|f(x,u_{k_j})-f(x,u)|^{\frac{2\tilde{\nu}}{\tilde{\nu}+2}}dx\right)^{\frac{\tilde{\nu}+2}{2\tilde{\nu}}}
 \left(\int_{\Omega}|v|^{\frac{2\tilde{\nu}}{\tilde{\nu}-2}}dx\right)^{\frac{\tilde{\nu}-2}{2\tilde{\nu}}}\\
&\leq C\|f(x,u_{k_j})-f(x,u)\|_{L^{\frac{2\tilde{\nu}}{\tilde{\nu}+2}}(\Omega)}.
\end{aligned}
\end{equation}
Consequently, $K(u_{k_{j}})\to K(u)$ in $H_{X}^{-1}(\Omega)$ as $j\to +\infty$, and   $K$ maps bounded sets in ${H_{X,0}^1}(\Omega)$ to relatively compact sets in $H_X^{-1}(\Omega)$.\par

  Hence, by \eqref{3-4}, \eqref{3-12} and \eqref{3-16}, we know the Fr\'{e}chet derivative of  $E$ can be decomposed into
\begin{equation}
\label{3-18}
  DE(u)=L(u)+K(u)\qquad\mbox{for any}~~u\in H_{X,0}^{1}(\Omega),
\end{equation}
where $L:{H_{X,0}^1}(\Omega)\mapsto {H_{X}^{-1}(\Omega)}$ is a bounded invertible linear map and the operator $K$ maps bounded sets in ${H_{X,0}^1}(\Omega)$ to relatively compact sets in $H_X^{-1}(\Omega)$.\par

 Finally, according to Proposition \ref{prop2-10}, we can see that the Proposition \ref{prop3-2} will be proved by showing that any (PS) sequence is bounded in ${H_{X,0}^1}(\Omega)$.\par

Let $\{u_m\}_{m=1}^{\infty}$ be a (PS) sequence of $E$. Then we have
\[ |E_{1}(u_{m})|\leq C\qquad \mbox{for all}~~ m\geq 1~~~\mbox{and}~~ \|DE_{1}(u_m)\|_{H_X^{-1}(\Omega)}\to 0~~~\mbox{as}~~~m\to+\infty.\]
Thus, owing to assumption $(f_{3})$, we obtain
\begin{equation}
\label{3-19}
\begin{aligned}
&qC+\|DE(u_m)\|_{H_X^{-1}(\Omega)}\|u_m\|_{{H_{X,0}^1}(\Omega)}\\
&\geq qE(u_m)-\langle DE(u_m),u_m\rangle\\
&=\frac{q-2}{2}\int_{\Omega}|Xu_m|^2dx+\int_{\Omega}(f(x,u_m)u_m-qF(x,u_m))dx\\
&\geq \frac{q-2}{2}\int_{\Omega}|Xu_m|^2dx-\int_{\Omega_{1}}|f(x,u_m)u_m-qF(x,u_m)|dx\\
&\geq \frac{q-2}{2}\cdot \frac{\lambda_{1}}{1+\lambda_{1}} \|u_m\|_{H_{X,0}^1(\Omega)}^2-|\Omega|\max_{x\in{\overline{\Omega}},|v|\leq R_0}|f(x,v)v-qF(x,v)|,
\end{aligned}
\end{equation}
where $\Omega_1=\{x\in \Omega||u_m|\leq R_0\}$. Since $(f(x,v)v-qF(x,v))\in C(\overline{\Omega}\times \mathbb{R})$, we know the last term in \eqref{3-19} is finite. Thus, by \eqref{3-19} and Young's inequality, we can conclude $\{u_m\}_{m=1}^{\infty}$ is bounded in $H_{X,0}^{1}(\Omega)$.
\end{proof}

\begin{proposition}
\label{prop3-3}
If $f(x,u)$ satisfies the assumptions $(f_{1}),(f_{2})$ and $(f_{3})$, then there exist $\rho>0$ and $\alpha>0$, such that
\begin{enumerate}
  \item [(1)] $E(u)\geq \alpha$, for any $u\in H_{X,0}^{1}(\Omega)$ with $\|u\|_{H_{X,0}^{1}(\Omega)}=\rho$.
  \item [(2)] There exists a $u_{1}\in H_{X,0}^{1}(\Omega)$ such that $\|u_{1}\|_{H_{X,0}^{1}(\Omega)}>\rho$ and $E(u_{1})\leq 0$.
\end{enumerate}

\end{proposition}

\begin{proof}

The assumption $(f_{1})$ implies that, for any $\varepsilon>0$, there exists $\delta=\delta(\varepsilon)>0$ such that for $0<|u|<\delta$, $\left|\frac{f(x,u)}{u}\right|<\varepsilon$. Then, for $|u|\geq\delta$  (i.e. $1\leq \left(\frac{|u|}{\delta}\right)^{p-1}$), the assumption $(f_{2})$ yields that $$|f(x,u)|\leq C(1+|u|^{p-1})\leq C(\varepsilon)|u|^{p-1}$$
holds for all $x\in \overline{\Omega}$ with some constant $C(\varepsilon)>0$. Hence, we have
\begin{equation}\label{3-20}
|f(x,u)|\leq \varepsilon|u|+C(\varepsilon)|u|^{p-1},~~~\forall u\in{\mathbb{R}},~ x\in \overline{\Omega},
\end{equation}
and
\begin{equation}\label{3-21}
|F(x,u)|=\left|\int_{0}^{u}f(x,v)dv \right|\leq \varepsilon|u|^{2}+C(\varepsilon)|u|^{p},~~~\forall u\in{\mathbb{R}},~ x\in \overline{\Omega}.
\end{equation}
Thus, it follows from \eqref{3-21}, \eqref{2-5} and Proposition \ref{Poincare} that, there exist some constants $\varepsilon\in (0,\frac{\lambda_{1}}{4})$ and $\alpha>0$ such that
\begin{equation}\label{3-22}
\begin{aligned}
E(u)&=\frac{1}{2}\int_{\Omega}|Xu|^2dx-\int_{\Omega}F(x,u)dx\\
&\geq \frac{1}{2}\int_{\Omega}|Xu|^2dx-\varepsilon\int_{\Omega}|u|^2dx-C(\varepsilon)\int_{\Omega}|u|^pdx\\
&\geq \left(\frac{1}{2}- \frac{\varepsilon}{\lambda_1}\right)\int_{\Omega}|Xu|^2dx-C(\varepsilon)\int_{\Omega}|u|^pdx\\
&\geq \frac{\lambda_{1}}{1+\lambda_{1}}\left(\frac{1}{2}-\frac{\varepsilon}{\lambda_1}-
\widetilde{C}(\varepsilon)\|u\|^{p-2}_{{H_{X,0}^1}(\Omega)}\right)\|u\|^{2}_{{H_{X,0}^1}(\Omega)}
\geq\alpha>0,
\end{aligned}
\end{equation}
provided $\|u\|_{{H_{X,0}^1}(\Omega)}=\rho$ is sufficiently small. Here $\lambda_1>0$ is the first Dirichlet eigenvalue of $-\triangle_X$ on $\Omega$.\par

Next, the assumption $(f_{3})$ implies that for $|u|\geq R_0$ and all $x\in \overline{\Omega}$,
\begin{equation}\label{3-23}
u|u|^q\frac{\partial}{\partial u}(|u|^{-q}F(x,u))=f(x,u)u-qF(x,u)\geq 0.
\end{equation}
Then,  \eqref{3-23} gives that
\begin{equation}\label{3-24}
  F(x,u)\geq\gamma_0(x)|u|^q
\end{equation}
holds for $|u|\geq R_{0}$ and all $x\in \overline{\Omega}$ with $\gamma_0(x)=R_0^{-q}\min{\{F(x,R_0),F(x,-R_0)\}}>0$. Since $F(x,u)\in C(\overline{\Omega}\times\mathbb{R})$, there exists a constant $c>0$ such that
\begin{equation}\label{3-25}
\gamma_0(x)=R_0^{-q}\min{\{F(x,R_0),F(x,-R_0)\}}\geq c>0~~\mbox{for all}~~x\in \overline{\Omega}.
\end{equation}
Therefore, for any fixed $u_{0}\in {H_{X,0}^1}(\Omega)$ with $\|u_{0}\|_{{H_{X,0}^1}(\Omega)}\neq0$ and any $\lambda>0$, it follows from \eqref{3-24} and \eqref{3-25} that
\begin{equation}\label{3-26}
\begin{aligned}
E(\lambda u_{0})&=\frac{\lambda^2}{2}\int_{\Omega}|Xu_{0}|^2dx-\int_{\Omega}F(x,\lambda u_{0})dx\\
&\leq\frac{\lambda^2}{2}\|u_{0}\|_{H_{X,0}^{1}(\Omega)}^2-\int_{|\lambda u_{0}|\geq R_0}F(x,\lambda u_{0})dx+\int_{|\lambda u_{0}|\leq R_0}|F(x,\lambda u_{0})|dx\\
&\leq\frac{\lambda^2}{2}\|u_{0}\|_{H_{X,0}^{1}(\Omega)}^2-c\cdot\lambda^q\int_{|\lambda u_{0}|\geq R_0}|u_{0}|^qdx+|\Omega|\cdot\sup\limits_{x\in{\overline{\Omega}},|v|\leq R_0}|F(x,v)|\\
&\to -\infty, ~\mbox{as}~\lambda\rightarrow {+\infty}.
\end{aligned}
\end{equation}
Hence, for sufficiently large $\lambda>0$, taking $u_1=\lambda u_{0}$, we can conclude from \eqref{3-26} that
\begin{equation}\label{3-27}
  \|u_{1}\|_{H_{X,0}^{1}(\Omega)}> \rho\qquad \mbox{and}~~~ E(u_{1})\leq 0.
\end{equation}

\end{proof}

\begin{proposition}
\label{prop3-4}
Suppose that $f(x,u)$ satisfies assumptions $(f_2)$ and $(f_3)$. Then, we have the following conclusions:
\begin{enumerate}
  \item [(1)] There exist $k_{0}\in \mathbb{N}^{+}$, $\rho>0$ and $\alpha>0$ such that for all $u\in V_{k_{0}}=\text{span}~\{\varphi_k|k\geq k_0\}$ with $\|u\|_{H_{X,0}^{1}(\Omega)}=\rho$, $E(u)\geq\alpha$.
  \item [(2)] For any finite dimensional subspace $W\subset H_{X,0}^{1}(\Omega)$, there is a constant $R=R(W)>0$ such that $E(u)\leq 0$ for $u \in W$ with $ \|u\|_{H_{X,0}^{1}(\Omega)}\geq R$.
\end{enumerate}
\end{proposition}
\begin{proof}
The assumption $(f_2)$ implies that
\begin{equation}
\label{3-28}
  |F(x,u)|=\left|\int_{0}^{u}f(x,v)dv\right|\leq \left|\int_{0}^{u}|f(x,v)|dv\right|\leq C(|u|+|u|^{p}).
\end{equation}
 Then, by \eqref{3-28} and Proposition \ref{Poincare}, for any $\varepsilon>0$ and $u\in V_{k_{0}}$ we obtain
\begin{equation}\label{3-29}
\begin{aligned}
E(u)&=\frac{1}{2}\int_{\Omega}|Xu|^2dx-\int_{\Omega}F(x,u)dx\\
&\geq\frac{1}{2}\int_\Omega|Xu|^2dx-C\int_\Omega|u|^pdx-C\int_\Omega|u|dx\\
& \geq\frac{1}{2}\int_\Omega|Xu|^2dx-C\left(\int_{\Omega}|u|^{2}dx\right)^{\frac{r}{2}}\left(\int_{\Omega}|u|^{2_{\tilde{\nu}}^{*}}dx \right)^{\frac{p-r}{2_{\tilde{\nu}}^{*}}} -C\int_\Omega\left(\frac{1}{4\varepsilon}|u|^{2}+\varepsilon\right)dx\\
&=\frac{\lambda_1}{2(1+\lambda_1)}\|u\|_{H_{X,0}^{1}(\Omega)}^2-C\|u\|_{L^2(\Omega)}^r\|u\|_{L^{2_{\tilde{\nu}}^*}(\Omega)}^{p-r}-\frac{C}{4\varepsilon}\|u\|_{L^2(\Omega)}^2-C\varepsilon|\Omega|,
\end{aligned}
\end{equation}
where $r$ is a positive constant such that $\frac{r}{2}+\frac{p-r}{2_{\tilde{\nu}}^*}=1$.

On the other hand, we have the following Rayleigh-Ritz formula
\begin{equation}\label{3-30}
  \lambda_{k}=\inf_{u\in H_{X,0}^{1}(\Omega), u\neq 0,~ u\perp\varphi_{1},\ldots,\varphi_{k-1}}\frac{\int_{\Omega}|Xu|^{2}dx}{\int_{\Omega}|u|^{2}dx}.
\end{equation}
Thus, for $u\in V_{k_{0}}=\text{span}~\{\varphi_k|k\geq k_0\}$,  \eqref{3-30} indicates that
\begin{equation}\label{3-31}
  \lambda_{k_{0}}\int_{\Omega}|u|^{2}dx \leq \int_{\Omega}|Xu|^{2}dx.
\end{equation}
Take $\varepsilon=C\left(1+\frac{1}{\lambda_{1}}\right)$ in \eqref{3-29}. It follows from \eqref{2-5}, \eqref{3-29} and \eqref{3-31} that

\begin{equation}\label{3-32}
\begin{aligned}
E(u)&\geq \frac{\lambda_1}{2(1+\lambda_1)}\|u\|_{H_{X,0}^{1}(\Omega)}^2-C\|u\|_{L^2(\Omega)}^r\|u\|_{L^{2_{\tilde{\nu}}^*}(\Omega)}^{p-r}-\frac{C}{4\varepsilon}\|u\|_{L^2(\Omega)}^2-C\varepsilon|\Omega|\\
&\geq\frac{\lambda_1}{4(1+\lambda_1)}\|u\|_{H_{X,0}^{1}(\Omega)}^2-C\|u\|_{L^2(\Omega)}^r\|u\|_{L^{2_{\tilde{\nu}}^*}(\Omega)}^{p-r}-C^{2}(1+\lambda_{1}^{-1})|\Omega|\\
&\geq \frac{\lambda_1}{4(1+\lambda_1)}\|u\|_{H_{X,0}^{1}(\Omega)}^2-C_{1}\lambda_{k_{0}}^{-\frac{r}{2}}\|u\|_{H_{X,0}^{1}(\Omega)}^p-C^{2}(1+\lambda_{1}^{-1})|\Omega|\\
&=\left(\frac{\lambda_1}{4(1+\lambda_1)}-C_{1}\lambda_{k_{0}}^{-\frac{r}{2}}\|u\|_{H_{X,0}^{1}(\Omega)}^{p-2}   \right)\|u\|_{H_{X,0}^{1}(\Omega)}^2-C^{2}(1+\lambda_{1}^{-1})|\Omega|\\
&=\left(\frac{\lambda_1}{8(1+\lambda_1)}-C_{1}\lambda_{k_{0}}^{-\frac{r}{2}}\|u\|_{H_{X,0}^{1}(\Omega)}^{p-2}   \right)\|u\|_{H_{X,0}^{1}(\Omega)}^2+\frac{\lambda_1}{8(1+\lambda_1)}\|u\|_{H_{X,0}^{1}(\Omega)}^2-C^{2}(1+\lambda_{1}^{-1})|\Omega|
\end{aligned}
\end{equation}
holds for all $u\in V_{k_{0}}$, where $C_{1}>0$ is a constant. Then, we can find a positive constant $\rho>0$ such that $\frac{\lambda_{1}}{8(1+\lambda_{1})}\rho^{2}=C^{2}(1+\lambda_{1}^{-1})|\Omega|+1$, and choose a $k_{0}\in \mathbb{N}^{+}$ such that $\frac{\lambda_1}{8(1+\lambda_1)}\geq C_{1}\lambda_{k_{0}}^{-\frac{r}{2}}\rho^{p-2}$. Therefore, we can deduce from \eqref{3-32} that
\[E(u)\geq 1=:\alpha,~\mbox{for all}~ u\in V_{k_{0}}~\mbox{with}~\|u\|_{H_{X,0}^{1}(\Omega)}=\rho. \]
This prove the conclusion (1) of Proposition \ref{prop3-4}.

Next, for any finite dimensional subspace $W\subset H_{X,0}^1(\Omega)$, if $u\in W$ such that $\|u\|_{H_{X,0}^{1}(\Omega)}=\rho>0$, we let $v=\frac{u}{\rho}\in W$ with $\|v\|_{H_{X,0}^{1}(\Omega)}=1$. An argument similar to \eqref{3-26} shows that
\begin{equation}\label{3-33}
\begin{aligned}
E(u)&=E(\rho v)=\frac{\rho^2}{2}\int_{\Omega}|Xv|^2dx-\int_{\Omega}F(x,\rho v)dx\\
&\leq\frac{\rho^2}{2}-c\cdot\rho^q\int_{|\rho v|\geq R_0}|v|^qdx+|\Omega|\cdot\sup\limits_{x\in{\overline{\Omega}},|w|\leq R_0}|F(x,w)|\\
&\to -\infty, ~\mbox{as}~\rho\rightarrow {+\infty},
\end{aligned}
\end{equation}
which yields the conclusion (2).

\end{proof}

 Now, we give the proof of  Theorem \ref{thm1-1}.
\begin{proof}[Proof of Theorem \ref{thm1-1}] Clearly,
  \eqref{3-3} gives $E(0)=0$. If $f(x,u)$ satisfies assumptions $(f_{1})$-$(f_{3})$, it follows from Proposition \ref{prop3-1}, Proposition \ref{prop3-2}, Proposition \ref{prop3-3} and  Mountain Pass Theorem (Proposition \ref{prop2-11}) that the functional $E$ has a positive critical value, which implies the semilinear subelliptic Dirichlet problem \eqref{problem3-1} has a nontrivial weak solution in $H_{X,0}^{1}(\Omega)$.

If $f(x,u)$ satisfies the assumptions $(f_2)$-$(f_4)$, by Proposition \ref{prop3-1}, Proposition \ref{prop3-2}, Proposition \ref{prop3-4} and Symmetrical Mountain Pass Theorem (Proposition \ref{prop2-12}), we know the functional $E$ has an unbounded sequence of critical values $\{E(u_{k})\}_{k=1}^{\infty}$. Then, from \eqref{3-3}, \eqref{3-28} and \eqref{2-5}, we can deduce that there exist some positive constant $C$ such that
\begin{equation}\label{3-34}
 |E(u)|\leq C\left(\|u\|_{H_{X,0}^{1}(\Omega)}+\|u\|_{H_{X,0}^{1}(\Omega)}^{2}+\|u\|_{H_{X,0}^{1}(\Omega)}^{p} \right)~~\mbox{for all}~~u\in H_{X,0}^{1}(\Omega).
\end{equation}
That means problem \eqref{problem3-1} admits an unbounded sequence $\{u_k\}_{k=1}^{\infty}$ of weak solutions in $H_{X,0}^1(\Omega)$.\par
\end{proof}

\section{Multiplicity of weak solutions with perturbation term}
\label{Section4}
   Now, we study the problem \eqref{problem1-1} with non-vanishing perturbation term $g$.

\begin{definition}
\label{def4-1}
For $g\in L^2(\Omega)$, we say that $u\in H_{X,0}^{1}(\Omega)$ is a weak solution of \eqref{problem1-1} if
\begin{equation}
\label{4-1}
  \int_{\Omega}Xu\cdot Xvdx-\int_{\Omega}f(x,u)vdx-\int_{\Omega}g(x)vdx=0,~~~\mbox{for all}~v\in H_{X,0}^{1}(\Omega).
\end{equation}
\end{definition}

Then, we consider the following energy functional $E_{1}: H_{X,0}^{1}(\Omega)\to \mathbb{R}$, defined as
\begin{equation}\label{4-2}
  E_{1}(u)=\frac{1}{2}\int_{\Omega}{{|{Xu}|}^2} dx-\int_{\Omega}F(x,u) dx-\int_{\Omega}gudx.
\end{equation}
 By the similar arguments of Proposition \ref{prop3-1},   we have
\begin{proposition}
\label{prop4-1}
If $f(x,u)$ satisfies the assumption $(f_{2})$ and $g\in L^2(\Omega)$, then functional
  \[ E_{1}(u)=\frac{1}{2}\int_{\Omega}{{|{Xu}|}^2} dx-\int_{\Omega}F(x,u) dx-\int_{\Omega}gudx \]
belongs to $C^{1}(H_{X,0}^{1}(\Omega), \mathbb{R})$. Thus the semilinear equation \eqref{problem1-1} is the Euler-Lagrange equation of the variational problem for the energy functional \eqref{4-2}. Furthermore, the Fr\'{e}chet derivative of $E_{1}$ at $u$ is given by
\begin{equation}\label{4-3}
  \langle DE_{1}(u),v\rangle=\int_{\Omega}Xu\cdot Xvdx-\int_{\Omega}f(x,u)vdx-\int_{\Omega}gvdx,~~\mbox{for all}~~v\in H_{X,0}^{1}(\Omega).
\end{equation}
Therefore, the critical point of $E_{1}$ in $H_{X,0}^{1}(\Omega)$ is the weak solution to \eqref{problem1-1}.
\end{proposition}

\begin{proof}
 Proposition \ref{prop3-1} indicates that $E\in C^{1}(H_{X,0}^{1}(\Omega), \mathbb{R})$. Observe that
\begin{equation}\label{4-4}
 E_{1}(u)=E(u)-\int_{\Omega}gudx~~\mbox{for all}~~u\in H_{X,0}^{1}(\Omega),
\end{equation}
and the functional $I_{1}(u)=\int_{\Omega}gudx$ admits the continuous G\^{a}teaux derivative
\begin{equation}\label{4-5}
  \langle I_{1}'(u),v\rangle=\int_{\Omega}gvdx~~\mbox{for all}~~ v\in H_{X,0}^{1}(\Omega).
\end{equation}
 Therefore,  $E_{1}\in C^{1}(H_{X,0}^{1}(\Omega), \mathbb{R})$.
\end{proof}

Inspired by Rabinowitz's approach in \cite{Rabinowitz1986}, we shall then construct a new functional $E_2$ which is a modification of $E_{1}$ such that large critical values and points of $E_2$ are critical values and points of $E_1$. Then, the conclusion of Theorem \ref{thm1-2} follows if we prove that $E_{2}$ has an unbounded sequence of critical values.\par

Firstly, from \eqref{3-24} and \eqref{3-25}, the assumption $(f_{3})$ implies there exist constants $a_2, a_3>0$ such that
\begin{equation}\label{4-6}
F(x,u)\geq a_3|u|^q-a_2~~~\mbox{for all}~~u\in \mathbb{R},~ x\in \overline{\Omega}.
\end{equation}
Hence, there is a constant $a_1>0$, such that
\begin{equation}\label{4-7}
\frac{1}{q}\left(uf(x,u)+a_1\right)\geq F(x,u)+a_2\geq a_3|u|^q~~~\mbox{for all}~~u\in \mathbb{R},~x\in \overline{\Omega}.
\end{equation}

\begin{proposition}
\label{prop4-2}
Under the hypotheses of Theorem \ref{thm1-2}, there is a positive constant $A$ depending on $\|g\|_{L^2(\Omega)}$ such that if $u$ is a critical point if $E_1$, then
\begin{equation}
\label{4-8}
\int_{\Omega}(F(x,u)+a_2)dx\leq A(E_1(u)^2+1)^{\frac{1}{2}}.
\end{equation}
\end{proposition}
\begin{proof}
Suppose $u$ is a critical point of $E_1$. Then by \eqref{4-7} and simple estimates,
\begin{equation}
\label{4-9}
\begin{aligned}
&E_1(u)=E_1(u)-\frac{1}{2}\langle DE_1(u), u\rangle=\int_{\Omega}\left(\frac{1}{2}uf(x,u)-F(x,u)-\frac{1}{2}gu\right)dx\\
&\geq \left(\frac{1}{2}-\frac{1}{q}\right)\int_{\Omega}(uf(x,u)+a_1)dx-\frac{1}{2}\|g\|_{L^2(\Omega)}\|u\|_{L^2(\Omega)}-a_4\\
&\geq a_5 \int_{\Omega}(F(x,u)+a_2)dx-a_6\|u\|_{L^q(\Omega)}-a_4
\geq \frac{a_5}{2}\int_{\Omega}(F(x,u)+a_2)dx-a_7
\end{aligned}
\end{equation}
and \eqref{4-8} follows immediately from \eqref{4-9}.
\end{proof}

Suppose $\chi \in C^\infty(\mathbb{R}, \mathbb{R})$ such that $\chi(\xi)\equiv 1$ for $\xi\leq 1$, $\chi(\xi)\equiv 0$ for $\xi\geq 2$, and $\chi'(\xi)\in (-2,0)$ for $\xi \in (1,2)$. Let
\[Q(u)=2A(E_1(u)^2+1)^{\frac{1}{2}}\]
and
\[\psi(u)=\chi\left(Q(u)^{-1}\int_{\Omega}(F(x,u)+a_2)dx \right).\]
Note that by \eqref{4-8}, if $u$ is a critical point of $E_1$,  $Q(u)^{-1}\int_{\Omega}(F(x,u)+a_2)dx$ lies in $[0,\frac{1}{2}]$ and then $\psi(u)=1$. Finally, we let
\begin{equation}
\label{4-10}
E_2(u)=\int_{\Omega}\left(\frac{1}{2}{|{Xu}|}^2-F(x,u)-\psi(u)gu\right)dx.
\end{equation}
Then $E_2(u)=E_1(u)$ if $u$ is a critical point of $E_{1}$.\par

The following proposition contains the main technical properties of $E_{2}$ which we need.

\begin{proposition}
\label{prop4-3}
Under the hypotheses of Theorem \ref{thm1-2},  we can obtain that
\begin{enumerate}
  \item [(1)] $E_2 \in C^1(H_{X,0}^1(\Omega), \mathbb{R})$.
  \item [(2)] There exists a positive constant $B_1$ depending on $\|g\|_{L^2(\Omega)}$ such that
  \begin{equation}
  \label{4-11}
  |E_2(u)-E_2(-u)|\leq B_1(|E_2(u)|^{\frac{1}{q}}+1)~~~~\mbox{for all}~~u\in H_{X,0}^1(\Omega).
  \end{equation}
  \item [(3)] There is a constant $M_0>0$ such that if $E_2(u)\geq M_0$ and $DE_2(u)=0$, then $E_2(u)=E_1(u)$ and $DE_1(u)=0$.
  \item [(4)] There is a constant $M_1\geq M_0$ such that $E_2$ satisfies $(PS)$ condition on $\widehat{A}_{M_{1}}=\{u\in H_{X,0}^{1}(\Omega)|E_{2}(u)\geq M_{1}\}$.
\end{enumerate}
\end{proposition}

\begin{proof}
From \eqref{4-4} and \eqref{4-10}, we have
\begin{equation}\label{4-12}
 E_{2}(u)=E(u)-\psi(u)\int_{\Omega}gudx=E_{1}(u)+(1-\psi(u))\int_{\Omega}gudx~~\mbox{for all}~~u\in H_{X,0}^{1}(\Omega).
\end{equation}
Since $\chi$ is smooth, then $\psi\in C^1(H_{X,0}^1(\Omega), \mathbb{R})$ and therefore $E_2\in C^1(H_{X,0}^1(\Omega), \mathbb{R})$.

To prove \eqref{4-11}, for $u\in \mbox{supp}~ \psi$, we first show that there exists a constant  $\alpha_1$ depending on $\|g\|_{L^2(\Omega)}$ such that
\begin{equation}
\label{4-13}
\left|\int_{\Omega}gu dx\right|\leq \alpha_1(|E_1(u)|^{\frac{1}{q}}+1).
\end{equation}
Indeed, the H\"{o}lder's inequality and \eqref{4-7} yield that
\begin{equation}\label{4-14}
\begin{aligned}
\left|\int_{\Omega}gu dx\right|\leq \|g\|_{L^2(\Omega)}\|u\|_{L^2(\Omega)}\leq \alpha_2\|u\|_{L^q(\Omega)}
\leq \alpha_3\left(\int_{\Omega} (F(x,u)+a_2) dx\right)^{\frac{1}{q}}.
\end{aligned}
\end{equation}
Moreover, for $u\in \text{supp}~ \psi$, by the definition of $\psi(u)$, we get
\begin{equation}
\label{4-15}
\int_{\Omega} (F(x,u)+a_2) dx\leq 4A(E_1(u)^2+1)^{\frac{1}{2}}\leq \alpha_4(|E_1(u)|+1).
\end{equation}
Thus, \eqref{4-13} follows from \eqref{4-14} and \eqref{4-15}. Then, to obtain \eqref{4-11},  we can deduce from \eqref{4-10} and assumption $(f_4)$ that
\begin{equation}
\label{4-16}
|E_2(u)-E_2(-u)|\leq (\psi(u)+\psi(-u))\left|\int_{\Omega}gu dx\right|.
\end{equation}
Next we estimate the right-hand side of \eqref{4-16}. By \eqref{4-12},
\begin{equation}
\label{4-17}
|E_1(u)|\leq |E_2(u)|+\left|\int_{\Omega}gu dx\right|.
\end{equation}
Besides, \eqref{4-13} implies that, for any $u\in H_{X,0}^{1}(\Omega)$,
\begin{equation}\label{4-18}
\psi(u)\left|\int_{\Omega}gu dx\right|\leq \alpha_1\psi(u)(|E_1(u)|^{\frac{1}{q}}+1).
\end{equation}
Combining \eqref{4-17} and \eqref{4-18}, we can obtain
\begin{equation}\label{4-19}
\psi(u)\left|\int_{\Omega}gu dx\right|\leq \alpha_5\psi(u)\left(|E_2(u)|^{\frac{1}{q}}+\left|\int_{\Omega}gu dx\right|^{\frac{1}{q}}+1\right).
\end{equation}
Hence, by using Young's inequality, the $|\int_{\Omega}gudx|^{\frac{1}{q}}$ term on the right-hand side in \eqref{4-19} can be absorbed into the left-hand side leading
\begin{equation}\label{4-20}
\psi(u)\left|\int_{\Omega}gu dx\right|\leq \alpha_6(|E_2(u)|^{\frac{1}{q}}+1).
\end{equation}
The $\psi(-u)$ term of \eqref{4-16} can also be estimated by analogous argument above and thus \eqref{4-11} follows.

To prove conclusion $(3)$, it suffices to show that if $M_0$ is large and $u$ is a critical point of $E_2$ with $E_2(u)\geq M_0$, then
\begin{equation}\label{4-21}
Q(u)^{-1}\int_{\Omega}(F(x,u)+a_2) dx<1.
\end{equation}
The definition of $\psi$  and \eqref{4-21} imply $\psi\equiv1$ in a neighborhood of $u$. Hence, $D\psi(u)=0$ and then $E_2(u)=E_1(u)$, $DE_2(u)=DE_1(u)$, which yields conclusion $(3)$ .\par

We will show that  \eqref{4-21} holds. From  \eqref{4-10}, for any $u,v\in H_{X,0}^1(\Omega)$,
\begin{equation}\label{4-22}
\langle DE_2(u), v\rangle=\int_{\Omega}\left(Xu\cdot Xv-f(x,u)v-\langle D\psi(u), v\rangle gu-\psi(u)gv\right) dx,
\end{equation}
where
\begin{equation}\label{4-23}
\langle D\psi (u), v\rangle=\chi'(\theta(u))Q(u)^{-2}\left(Q(u)\int_{\Omega}f(x,u)v dx-(2A)^2\theta(u)E_1(u)\langle DE_1(u), v\rangle\right)
\end{equation}
and
\begin{equation}\label{4-24}
\theta(u)=Q(u)^{-1}\int_{\Omega}(F(x,u)+a_2) dx.
\end{equation}
Let
\begin{equation}\label{4-25}
T_1(u)=\chi'(\theta(u))(2A)^2Q(u)^{-2}E_1(u)\theta(u)\int_{\Omega}gu dx,
\end{equation}
and
\begin{equation}\label{4-26}
T_2(u)=\chi'(\theta(u))Q(u)^{-1}\int_{\Omega}gu dx+T_1(u).
\end{equation}
Hence, it follows from \eqref{4-22}, \eqref{4-23}, \eqref{4-25} and \eqref{4-26} that
\begin{equation}\label{4-27}
\begin{aligned}
\langle DE_2(u), v\rangle=&(1+T_1(u))\int_{\Omega}Xu\cdot Xv dx\\
&-(1+T_2(u))\int_{\Omega}f(x,u)v dx-(\psi(u)+T_1(u))\int_{\Omega}gv dx.
\end{aligned}
\end{equation}

Suppose that $u\in H_{X,0}^{1}(\Omega)$ is a critical point of $E_{2}$. If $u\in \text{supp}~\psi$ with $\theta(u)<1$, then $T_1(u)=T_2(u)=0$ and $\psi(v)\equiv1 $ in a neighborhood of $u$. Hence, $u$ is also a critical point of $E_{1}$ and \eqref{4-8} gives \eqref{4-21}. For $u\notin \text{supp}~\psi$, we know that $\psi(v)\equiv 0 $ in a neighborhood of $u$ and $T_1(u)=T_2(u)=0$. Consider
\begin{equation}\label{4-28}
\begin{aligned}
E_{2}(u)&=E_2(u)-\frac{1}{2(1+T_1(u))}\langle DE_2(u), u\rangle
=\frac{1+T_2(u)}{2(1+T_1(u))}\int_{\Omega}f(x,u)udx\\&-\int_{\Omega}F(x,u)dx-\frac{\psi(u)+T_1(u)(2\psi(u)-1)}{2(1+T_1(u))}\int_{\Omega}gu dx.
\end{aligned}
\end{equation}
 Then, we can deduce from \eqref{4-12} and \eqref{4-28} that
\[ E_{1}(u)=E_{2}(u)-\int_{\Omega}gudx=\frac{1}{2}\int_{\Omega}f(x,u)udx-\int_{\Omega}F(x,u)dx-\int_{\Omega}gudx,\]
which also gives \eqref{4-21} by a similar estimate of \eqref{4-9}. In the case of $u\in \text{supp}~\psi$ with $1\leq \theta(u)\leq 2$, we first consider the situation that $T_1(u)$ and $T_2(u)$ are both small enough (e.g. $|T_{1}(u)|, |T_{2}(u)|\leq \frac{1}{2}$ and $ \frac{1+T_{2}(u)}{1+T_{1}(u)}>\frac{2}{q}$).
 Observe that  \eqref{4-12} gives
\begin{equation}\label{4-29}
\begin{aligned}
E_2(u)\leq |E_2(u)|\leq |E_1(u)|+(1-\psi(u))\left|\int_{\Omega}gudx\right|
\leq |E_1(u)|+\|g\|_{L^2(\Omega)}\|u\|_{L^2(\Omega)}.
\end{aligned}
\end{equation}
On the other hand, by \eqref{4-7} and \eqref{4-28}, we have
\begin{equation}\label{4-30}
\begin{aligned}
E_2(u)&\geq \left(\frac{1+T_2(u)}{2(1+T_1(u))}-\frac{1}{q}\right)\int_{\Omega}(uf(x,u)+a_1) dx-C(u)\|g\|_{L^2(\Omega)}\|u\|_{L^2(\Omega)}-a_8\\
&\geq \left(\frac{q(1+T_2(u))}{2(1+T_1(u))}-1\right)\int_{\Omega}(F(x,u)+a_2)dx-C(u)\|g\|_{L^2(\Omega)}\|u\|_{L^2(\Omega)}-a_8,
\end{aligned}
\end{equation}
where $$C(u)=\left|\frac{\psi(u)+T_1(u)(2\psi(u)-1)}{2(1+T_1(u))}\right|\leq \frac{3}{2}.$$
Combining \eqref{4-29} and \eqref{4-30}, we obtain
\begin{equation}\label{4-31}
\begin{aligned}
|E_1(u)|&\geq \left(\frac{q(1+T_2(u))}{2(1+T_1(u))}-1\right)\int_{\Omega}(F(x,u)+a_2)dx-
(C(u)+1)\|g\|_{L^2(\Omega)}\|u\|_{L^2(\Omega)}-a_8\\
&\geq \left(\frac{q(1+T_2(u))}{2(1+T_1(u))}-1\right)\int_{\Omega}(F(x,u)+a_2)dx-
\frac{5}{2}\|g\|_{L^2(\Omega)}\|u\|_{L^2(\Omega)}-a_8\\
&\geq \left(\frac{q(1+T_2(u))}{2(1+T_1(u))}-1\right)\int_{\Omega}(F(x,u)+a_2)dx-
C\|u\|_{L^q(\Omega)}-a_8.
\end{aligned}
\end{equation}
By \eqref{4-7} and Young's inequality, if $T_1(u)$ and $T_2(u)$ are both small enough such that $ \frac{1+T_{2}(u)}{1+T_{1}(u)}>\frac{2}{q}$,  we can also deduce \eqref{4-8} by similar approach of \eqref{4-9} with $A$ replaced by a larger constant which is smaller than $2A$. Thus \eqref{4-21} is also valid.

Therefore, it suffices to show that for $u\in \text{supp}~\psi$ with $1\leq \theta(u)\leq 2$, $T_1(u), T_2(u)\rightarrow 0$ as $M_0\rightarrow\infty$. By \eqref{4-13}, \eqref{4-25} and \eqref{4-26} (the definitions of $T_1$, $T_{2}$), we get
\begin{equation}\label{4-32}
\begin{split}
|T_1(u)|&\leq 4\alpha_1(|E_1(u)|^{\frac{1}{q}}+1)|E_1(u)|^{-1},\\
|T_2(u)|&\leq (4+A^{-1})\alpha_1(|E_1(u)|^{\frac{1}{q}}+1)|E_1(u)|^{-1}.
\end{split}
\end{equation}
From \eqref{4-12}, $E_1(u)+\left|\int_{\Omega}gu dx\right|\geq E_2(u)$.
Thus by \eqref{4-13}, one has
\begin{equation}\label{4-33}
E_1(u)+\alpha_1|E_1(u)|^{\frac{1}{q}}\geq E_2(u)-\alpha_1\geq \frac{M_0}{2},
\end{equation}
for $M_0$ large enough, e.g. $M_{0}\geq 2\alpha_{1}$. If $E_1(u)\leq 0$, \eqref{4-33} implies that
\begin{equation}\label{4-34}
\frac{\alpha_1^{\tilde{q}}}{\tilde{q}}+\frac{|E_{1}(u)|}{q}\geq \frac{M_0}{2}+|E_1(u)|,
\end{equation}
where $\frac{1}{\tilde{q}}+\frac{1}{q}=1$.
But if $M_0\geq \frac{2\alpha_1^{\tilde{q}}}{\tilde{q}}$, then \eqref{4-34} is invalid. Hence, we know that $E_1(u)>0$. Then \eqref{4-33} implies $E_1(u)\geq \frac{M_0}{4}$ or $E_1(u)\geq (\frac{M_0}{4\alpha_1})^q$. In both cases, we can  conclude that $E_1(u)\rightarrow\infty$ as $M_0\rightarrow\infty$. Then \eqref{4-32} implies $T_1(u), T_{2}(u)\rightarrow 0$ as $M_0\rightarrow\infty$.  Consequently, if we take  $M_{0}$ sufficiently large enough  such that $|T_{1}(u)|, |T_{2}(u)|\leq \frac{1}{2}$ and $ \frac{1+T_{2}(u)}{1+T_{1}(u)}>\overline{c}>\frac{2}{q}$ for some positive constant $\overline{c}$, then the conclusion $(3)$ follows.\par

Finally, we prove the conclusion $(4)$. It suffices to show that there exists $M_1> M_0$ such that if $\{u_m\}\subset H_{X,0}^1(\Omega)$ satisfies $M_1\leq E_2(u_m)\leq K$ and $DE_2(u_m)\rightarrow 0$, then $\{u_m\}$ is bounded in $H_{X,0}^1(\Omega)$ and has a convergent subsequence. For large $m$ (such that $\|DE_2(u_m)\|_{H_{X}^{-1}(\Omega)}<1$) and any $\rho>0$, it follows from \eqref{2-6}, \eqref{4-10} and \eqref{4-27} that
\begin{equation}\label{4-35}
\begin{aligned}
K+\rho\|u_m\|_{H_{X,0}^1(\Omega)}&\geq E_2(u_m)-\rho\langle DE_2(u_m), u_m\rangle\\
&\geq \frac{\lambda_1}{1+\lambda_1}\left(\frac{1}{2}-\rho(1+T_1(u_m))\right)\|u_m\|^2_{H_{X,0}^1(\Omega)}\\
&+\rho(1+T_2(u_m))\int_{\Omega}f(x,u_m)u_m dx-\int_{\Omega}F(x,u_m)dx\\
&+[\rho(\psi(u_m)+T_1(u_m))-\psi(u_m)]\int_{\Omega}gu_m dx.
\end{aligned}
\end{equation}
For $M_1>M_{0}$ sufficiently large and therefore $T_1,~T_2$ small (e.g. $|T_{1}(u_{m})|, |T_{2}(u_{m})|\leq \frac{1}{2}$ and $ \frac{1+T_{2}(u_m)}{1+T_{1}(u_m)}>\frac{2}{q}$), we can choose $\rho\in \left(\frac{1}{q},\frac{1}{2}\right)$ and $\varepsilon>0$ such that
\begin{equation}\label{4-36}
\frac{1}{2(1+T_1(u_m))}> \rho +\varepsilon > \rho-\varepsilon > \frac{1}{q(1+T_2(u_m))}.
\end{equation}
Then by \eqref{4-7}, \eqref{4-35}, \eqref{4-36} and assumption $(f_3)$,  we have
\begin{eqnarray}\label{4-37}
&&K+\rho\|u_m\|_{H_{X,0}^1(\Omega)}\nonumber\\
&&\geq
\frac{\varepsilon\lambda_1}{1+\lambda_1}(1+T_1(u_m))\|u_m\|^2_{H_{X,0}^1(\Omega)}+\left(\frac{1}{q}+\varepsilon(1+T_2(u_m))\right)\int_{|u_m|\geq R_0}f(x,u_m)u_m dx\nonumber\\
&&+\rho(1+T_2(u_m))\int_{|u_m|\leq R_0}f(x,u_m)u_m dx-\int_{\Omega}F(x,u_m)dx\nonumber\\
&&+[\rho(\psi(u_m)+T_1(u_m))-\psi(u_m)]\int_{\Omega}gu_m dx\nonumber\\
&&\geq \frac{\varepsilon\lambda_1}{2(1+\lambda_1)}\|u_m\|^2_{H_{X,0}^1(\Omega)}+
\left(\frac{1}{q}+\frac{\varepsilon}{2}\right)\int_{|u_m|\geq R_0}f(x,u_m)u_m dx\\
&&-\int_{|u_m|\geq R_0}F(x,u_m)dx-\alpha_2-\alpha_3\|u_m\|_{L^2(\Omega)}\nonumber\\
&&\geq \frac{\varepsilon\lambda_1}{2(1+\lambda_1)}\|u_m\|^2_{H_{X,0}^1(\Omega)}+\frac{\varepsilon q}{2}\int_{|u_m|\geq R_0}F(x,u_m)dx-\alpha_2-\alpha_3\|u_m\|_{L^2(\Omega)}\nonumber\\
&&\geq \frac{\varepsilon\lambda_1}{2(1+\lambda_1)}\|u_m\|^2_{H_{X,0}^1(\Omega)}-\alpha_3\|u_m\|_{H_{X,0}^{1}(\Omega)}
-\alpha_2,\nonumber
\end{eqnarray}
which yields that $\{u_m\}$ is bounded in $H_{X,0}^1(\Omega)$.

Our task now reduces to show that $\{u_m\}$ has a convergent subsequence in  $H_{X,0}^1(\Omega)$.
For any $g\in L^2(\Omega)$, we have
\begin{equation}\label{4-38}
\left|\int_{\Omega}gvdx\right|\leq \|g\|_{L^2(\Omega)} \|v\|_{L^2(\Omega)}\leq \|g\|_{L^2(\Omega)} \|v\|_{H_{X,0}^{1}(\Omega)},
\end{equation}
which implies the linear functional $T_{g}$ given by
\begin{equation}\label{4-39}
  \langle T_{g},v\rangle:=-\int_{\Omega}gvdx,~~\forall v\in H_{X,0}^{1}(\Omega)
\end{equation}
belongs to $H_{X}^{-1}(\Omega)$. Thus, it follows from \eqref{3-12}, \eqref{3-16}, \eqref{4-27} and \eqref{4-39} that the Fr\'{e}chet derivative of $E_{2}$ can be decomposed into
\begin{equation}\label{4-40}
 DE_{2}(u)=(1+T_1(u))L(u)+(1+T_2(u))K(u)+(\psi(u)+T_1(u))T_{g}
 \end{equation}
for all $u\in H_{X,0}^{1}(\Omega)$, where $L:{H_{X,0}^1}(\Omega)\mapsto {H_{X}^{-1}(\Omega)}$ is a bounded invertible linear map and the operators $K$ and $T_g$ map bounded sets in ${H_{X,0}^1}(\Omega)$ to relatively compact sets in $H_X^{-1}(\Omega)$.
Hence, \eqref{4-40} yields that
\begin{equation}\label{4-41}
L^{-1}DE_{2}(u_m)=(1+T_1(u_m))u_m+(1+T_2(u_m))L^{-1}K(u_m)+(\psi(u_m)+T_1(u_m))L^{-1}T_{g}.
\end{equation}
Furthermore, for sufficiently large $M_1$, the sequences $\{T_1(u_m)\}$, $\{T_2(u_m)\}$ and $\{\psi(u_m)\}$ are bounded.  Then there exists a subsequence $\{u_{m_k}\}\subset \{u_{m}\}$ such that
\begin{equation}\label{4-42}
\lim_{k\rightarrow\infty}T_1(u_{m_k})=c_1,~\lim_{k\rightarrow\infty}T_2(u_{m_k})=c_2,~\lim_{k\rightarrow\infty}\psi(u_{m_k})=c_3.
\end{equation}
Since  $\{u_{m_{k}}\}$ is also bounded in $H_{X,0}^1(\Omega)$, $L^{-1}K(u_{m_{k}})$ converges along a subsequence $\{u_{m_{k_{j}}}\}\subset \{u_{m_k}\}$. Therefore, by \eqref{4-41} and \eqref{4-42}, we obtain $\{u_{m_{k_j}}\}$ converges in $H_{X,0}^1(\Omega)$.

\end{proof}

From Proposition \ref{prop4-3} $(3)$, we know that, Theorem \ref{thm1-2} can be derived by showing that $E_2$ has an unbounded sequence of critical points. This goal will be obtained via several steps. First, we introduce a sequence of minimax values of $E_2$ and establish its lower bound estimate on the basis of Proposition \ref{prop2-5}.

For any finite dimensional subspace $W\subset H_{X,0}^1(\Omega)$, if $u\in W$ such that $\|u\|_{H_{X,0}^{1}(\Omega)}=\rho>0$, we let $v=\frac{u}{\rho}$. Then we obtain from assumption $(f_3)$ and \eqref{3-33} that
\begin{equation}\label{4-43}
\begin{aligned}
E_{2}(u)&=E_{2}(\rho v)=E(\rho v)-\rho\psi(\rho v)\int_{\Omega}gvdx\\
&=\frac{\rho^2}{2}\int_{\Omega}|Xv|^2dx-\int_{\Omega}F(x,\rho v)dx-\rho\psi(\rho v)\int_{\Omega}gvdx\\
&\leq\frac{\rho^2}{2}-c\cdot\rho^q\int_{|\rho v|\geq R_0}|v|^qdx+|\Omega|\cdot\sup\limits_{x\in{\overline{\Omega}},|w|\leq R_0}|F(x,w)|+\rho\|g\|_{L^2(\Omega)}\\
&\to -\infty, ~\mbox{as}~\rho\rightarrow {+\infty}.
\end{aligned}
\end{equation}
  The estimate \eqref{4-43}  indicates that for any finite dimensional subspace $W\subset H_{X,0}^{1}(\Omega)$, there is a constant $R=R(W)>0$ such that $E_{2}(u)\leq 0$ for $u \in W$ and $ \|u\|_{H_{X,0}^{1}(\Omega)}\geq R$. In particular, let $W_j=\text{span}\{\varphi_k|1\leq k\leq j\}$ and $W_j^{\perp}=\text{span}\{\varphi_k|k\geq j+1\}$ be the orthogonal complement of $W_j$ in $H_{X,0}^1(\Omega)$.
 We can choose an increasing sequence $\{R_j\}_{j=1}^{\infty}$ such that $R_j\geq \lambda_j^{\frac{r}{2(p-2)}}$ ($r=\tilde{\nu}(1-\frac{p}{2_{\tilde{\nu}}^*})$) and $E_2(u)\leq 0$ for all $u\in W_j$ with $\|u\|_{H_{X,0}^{1}(\Omega)}\geq R_j$.\par

Let $D_j=\overline{B_{R_j}}\cap W_j$ and
\[G_j=\{h\in C(D_j,H_{X,0}^{1}(\Omega))| ~h ~\mbox{is odd and} ~h=\mbox{\textbf{id} on }\partial B_{R_j}\cap W_j \},\]
where $B_{R}=\{u\in H_{X,0}^{1}(\Omega)|\|u\|_{H_{X,0}^{1}(\Omega)}<R\}$ and $\textbf{id}$ denotes the identity map. Clearly, $\textbf{id}\in G_{j}$. Define
\begin{equation}\label{4-44}
b_j=\inf\limits_{h\in G_j}\max\limits_{u\in D_j}E_2(h(u)),\quad j\in \mathbb{N}^{+}.
\end{equation}
Then, we have the following lower bound estimate for the minimax values $b_{j}$.

\begin{proposition}
\label{prop4-4}
There exist constants $C_{2}>0$ and $\tilde{k}\in \mathbb{N}$ such that
\begin{equation}\label{4-45}
  b_{k}\geq C_{2}\cdot k^{\frac{2p}{\tilde{\nu}(p-2)}-1}\qquad \mbox{for all}~~k\geq \tilde{k}.
\end{equation}
\end{proposition}

\begin{proof}
For $h \in G_k$ and $\rho<R_k$, by the Intersection Theorem (see Lemma 1.44 in \cite{Rabinowitz1982}), we know that $h(D_k)\cap \partial B_{\rho}\cap W_{k-1}^{\bot}\neq\varnothing$. Then

\begin{equation}\label{4-46}
\begin{aligned}
\max\limits_{u\in D_k}E_2(h(u))&=\max\limits_{u\in h(D_k)}E_2(u)\geq \max\limits_{u\in h(D_k)\cap \partial B_{\rho}\cap W_{k-1}^{\bot}}E_2(u)\\
&\geq \inf\limits_{u\in h(D_k)\cap \partial B_{\rho}\cap W_{k-1}^{\bot} }E_2(u)\geq\inf\limits_{u\in\partial B_{\rho}\cap W_{k-1}^{\bot} }E_2(u).
\end{aligned}
\end{equation}
Moreover, for any $\delta>0$ and $u\in W_{k-1}^{\bot}=V_{k}$, we have $|gu|\leq \frac{1}{4\delta}u^{2}+\delta g^2$. Thus, for any $\varepsilon>0$, the estimate \eqref{3-29} indicates that
\begin{equation}\label{4-47}
\begin{aligned}
E_2(u)&=\frac{1}{2}\int_{\Omega}|Xu|^2dx-\int_{\Omega}F(x,u)dx-\psi(u)\int_{\Omega}gu dx\\
&\geq \frac{1}{2}\int_{\Omega}|Xu|^2dx-\int_{\Omega}F(x,u)dx-\left|\int_{\Omega}gu dx\right|\\
      &\geq \frac{\lambda_1}{2(1+\lambda_1)}\|u\|_{H_{X,0}^{1}(\Omega)}^2-C\|u\|_{L^2(\Omega)}^r\|u\|_{L^{2_{\tilde{\nu}}^*}(\Omega)}^{p-r}-\left(\frac{C}{4\varepsilon}+\frac{1}{4\delta}\right)\|u\|_{L^2(\Omega)}^2\\
      &-\left(C\varepsilon|\Omega|+\delta\|g\|_{L^2(\Omega)}^2\right),
\end{aligned}
\end{equation}
where $r$ is a positive constant such that $\frac{r}{2}+\frac{p-r}{2_{\tilde{\nu}}^*}=1$. We can choose $\varepsilon$ and $\delta$ such that $\frac{C}{\varepsilon}=\frac{\lambda_{1}}{2(1+\lambda_{1})}= \frac{1}{\delta}$.
It follows from \eqref{2-5}, \eqref{3-30} and \eqref{4-47} that, for $u\in W_{k-1}^{\bot}\cap\partial B_{\rho}$,
\begin{equation}\label{4-48}
\begin{aligned}
E_2(u)&\geq\frac{\lambda_1}{4(1+\lambda_1)}\|u\|_{H_{X,0}^{1}(\Omega)}^2-C\|u\|_{L^2(\Omega)}^r\|u\|_{L^{2_{\tilde{\nu}}^*}(\Omega)}^{p-r}-2C^{2}(1+\lambda_{1}^{-1})|\Omega|-2\|g\|^2_{L^2(\Omega)}(1+\lambda_{1}^{-1})\\
&\geq \frac{\lambda_1}{4(1+\lambda_1)}\|u\|_{H_{X,0}^{1}(\Omega)}^2-C_{1}\lambda_{k}^{-\frac{r}{2}}\|u\|_{H_{X,0}^{1}(\Omega)}^p-2C^{2}(1+\lambda_{1}^{-1})|\Omega|-2\|g\|^2_{L^2(\Omega)}(1+\lambda_{1}^{-1})\\
&=\left(\frac{\lambda_1}{4(1+\lambda_1)}-C_{1}\lambda_{k}^{-\frac{r}{2}}\rho^{p-2}   \right)\rho^2
-2C^{2}(1+\lambda_{1}^{-1})|\Omega|-2\|g\|^2_{L^2(\Omega)}(1+\lambda_{1}^{-1})\\
&\geq C_{0}\lambda_{k}^{\frac{r}{p-2}}=C_{0}\lambda_k^{(\frac{p}{p-2}-\frac{\tilde{\nu}}{2})}~~~\mbox{for all}~~k\geq \tilde{k},
\end{aligned}
\end{equation}
where $r=\tilde{\nu}(1-\frac{p}{2_{\tilde{\nu}}^*})=p-\frac{\tilde{\nu}(p-2)}{2}$ and $C_{0}>0$ is a constant. The last step in \eqref{4-48} is obtained by taking $\rho=\varepsilon_0\lambda_k^{\frac{r}{2(p-2)}}$, where $0<\varepsilon_0<\min\left\{1,(\frac{\lambda_1}{8C_1(1+\lambda_1)})^{\frac{1}{p-2}}\right\}$.
Thus, combining \eqref{4-44}, \eqref{4-46}, \eqref{4-48} and \eqref{2-9}, we get
\begin{equation}\label{4-49}
  b_{k}\geq C_{2}\cdot k^{\frac{2p}{\tilde{\nu}(p-2)}-1}\qquad \mbox{for all}~~k\geq \tilde{k},
\end{equation}
where $C_{2}>0$ and $\tilde{k}$ is a positive integer.
\end{proof}

 Next, we  introduce another kind of minimax values to get critical values of $E_2$.
Let
\[U_k=\{u=t\varphi_{k+1}+w|~t\in [0, R_{k+1}], w\in \overline{B_{R_{k+1}}}\cap W_k, \|u\|_{H_{X,0}^{1}(\Omega)}\leq R_{k+1}\}\]
and
\begin{align*}
\Lambda_k=\{&H\in C(U_k, H_{X,0}^{1}(\Omega))|~H|_{D_k}\in G_k ~\mbox{and}~H=\mbox{\textbf{id} for }\\
&u\in Q_k=(\partial B_{R_{k+1}}\cap W_{k+1})\cup ((B_{R_{k+1}}\setminus B_{R_{k}})\cap W_k)\}.
\end{align*}
Define
\[c_k=\inf\limits_{H\in \Lambda_k}\max\limits_{u\in U_k}E_2(H(u)).\]
Observe that for any $H\in \Lambda_k$, since $D_{k}\subset U_{k}$ and $H|_{D_{k}}\in G_{k}$, then we have
\[ \max_{u\in U_{k}}E_{2}(H(u))\geq \max_{u\in D_{k}}E_{2}(H(u))=\max_{u\in D_{k}}E_{2}(H|_{D_{k}}(u))\geq \inf_{h\in G_{k}}\max_{u\in D_{k}}E_{2}(h(u)). \]
That means $c_k\geq b_k$.

\begin{proposition}
\label{prop4-5}
Assume $c_k>b_k\geq M_1$. For $\delta\in (0, c_k-b_k)$, define
\[\Lambda_k(\delta)=\{H\in \Lambda_k|~E_2(H(u))\leq b_k+\delta ~\mbox{for}~u\in D_k\}\]
and
\[c_k(\delta)=\inf\limits_{H\in \Lambda_k(\delta)}\max\limits_{u\in U_k}E_2(H(u)).\]
Then $c_k(\delta)$ is a critical value of $E_2$.

\end{proposition}

\begin{proof}
The definition of $\Lambda_k(\delta)$ implies this set is nonempty. Indeed, the definition of $b_k$ implies for $\delta\in (0, c_k-b_k)$, there exists a $h\in G_k$ such that $E_2(h(u))\leq b_k+\delta$ holds for all $u\in D_k$. Consider the following continuous extension of $h$ on $U_{k}$, i.e.
\begin{equation}\label{4-50}
 H(u)=\left\{\begin{array}{ll}{u,} & { u\in U_k\backslash D_k} \\[2mm] {h(u),} & {u\in D_k.}\end{array}\right.
\end{equation}
That means, $H \in \Lambda_k(\delta)$.\par

Since $\Lambda_{k}(\delta)\subset \Lambda_{k}$, $c_{k}(\delta)\geq c_{k}$. Suppose $c_k(\delta)$ is not a critical value of $E_2$. Let $\bar{\varepsilon}=\frac{1}{2}(c_k-b_k-\delta)>0$.
Then by Proposition \ref{prop4-3} (4) and Deformation Theorem (see Lemma 1.60 in \cite{Rabinowitz1982} and Theorem A.4 in \cite{Rabinowitz1986}), for $c_k(\delta)>M_{1}$ and $\bar{\varepsilon}>0$, there exist  $\varepsilon\in (0, \bar{\varepsilon})$ and $\eta\in C([0,1]\times H_{X,0}^{1}(\Omega), H_{X,0}^{1}(\Omega))$ such that
\begin{equation}\label{4-51}
\eta(t,u)=u ~\mbox{for all }~t\in [0,1]~ ~\mbox{if}~~E_2(u)\notin [c_k(\delta)-\bar{\varepsilon}, c_k(\delta)+\bar{\varepsilon}],
\end{equation}
and
\begin{equation}\label{4-52}
\eta(1, A_{c_k(\delta)+\varepsilon})\subset A_{c_k(\delta)-\varepsilon},
\end{equation}
where $A_c=\{u\in H_{X,0}^{1}(\Omega)|~E_2(u)\leq c\}$.
By the definition of $c_k(\delta)$, we can choose a $H\in \Lambda_k(\delta)$ such that
\begin{equation}\label{4-53}
\max\limits_{u\in U_k}E_2(H(u))\leq c_k(\delta)+\varepsilon.
\end{equation}
We then claim that $\eta (1, H(\cdot))\in \Lambda_k(\delta)$. Clearly, $\eta (1, H(\cdot))\in C(U_k, H_{X,0}^{1}(\Omega))$.
Since $H\in \Lambda_k$, if $u\in Q_k$, $H(u)=u$ and therefore $E_2(H(u))=E_2(u)\leq 0$ according to the definitions of $R_k$ and $R_{k+1}$. Recall that $b_k\geq M_1>0$ and $c_k(\delta)\geq c_k>b_k$, we have $c_{k}(\delta)-\bar{\varepsilon}=c_{k}(\delta)-\frac{1}{2}c_{k}+\frac{1}{2}b_{k}+\frac{1}{2}\delta>0$. Owing to \eqref{4-51}, we get  $\eta(1, H(u))=H(u)=u$ on $Q_k$. Then, we show that $\eta(1,H(\cdot))|_{D_{k}}\in G_{k}$.
For any $u\in D_k$, $E_2(H(u))\leq b_k+\delta$ due to $H\in \Lambda_k(\delta)$ and therefore
\[E_2(H(u))-\frac{1}{2}(b_k+\delta)\leq \frac{1}{2}(b_k+\delta)<\frac{1}{2}c_k\leq c_k(\delta)-\frac{1}{2}c_k,\]
which means
\begin{equation}\label{4-54}
E_2(H(u))< c_k(\delta)-\frac{1}{2}(c_k-b_k-\delta)=c_k(\delta)-\bar{\varepsilon}
\end{equation}
and then $\eta(1, H(u))=H(u)$.
We can also get $\eta(1, H(-u))=H(-u)=-H(u)$. Therefore, $\eta(1, H(\cdot))$ is odd on $ D_k$. Moreover, for any $u\in \partial B_{R_k}\cap W_k$, we have $H(u)=u$ and $E_{2}(u)\leq 0$, which implies $\eta(1, H(u))=u$ on $\partial B_{R_k}\cap W_k$.  Hence,  $\eta (1, H(\cdot))|_{D_k}\in G_k$ and $\eta (1, H(\cdot))\in \Lambda_k$. Furthermore, the arguments above indicate that, for $u\in D_k$, $E_2(\eta(1, H(u)))=E_2(H(u))\leq b_k+\delta$. Consequently, we obtain $\eta (1, H(\cdot))\in \Lambda_k(\delta)$.
Besides, from \eqref{4-53}, we know that $H(u)\in A_{c_{k}(\delta)+\varepsilon}$ for all $u\in U_{k}$. Thus, \eqref{4-52} yields that
\[\max_{u\in U_k}E_2(\eta(1, H(u)))\leq c_k(\delta)-\varepsilon,\]
which is contrary to the definition of $c_k(\delta)$.

\end{proof}

Finally, we show that, $c_k>b_k$ holds for infinitely
many $k$.
\begin{proposition}
\label{prop4-6}
For any positive integer $\bar{k}$ satisfying $\bar{k}\geq\tilde{k}$ (where $\tilde{k}$ is the same constant given in \eqref{4-45}), if $c_k=b_k$ for all $k\geq \bar{k}$, then there exists a constant $\tilde{M}>0$  such that
\begin{equation}\label{4-55}
b_k\leq \tilde{M}.
\end{equation}
Hence, from Proposition \ref{prop4-4}, we know that $c_k>b_k$ must hold for infinitely
many $k$.

\end{proposition}

\begin{proof}
For $\varepsilon>0$ and $k\geq \bar{k}$, the definition of $c_{k}$ allows us to choose a $H\in \Lambda_k$ such that
\begin{equation}\label{4-56}
\max\limits_{u\in U_k}E_2(H(u))\leq c_k+\varepsilon=b_k+\varepsilon.
\end{equation}
Observe that $D_{k+1}=U_k\cup (-U_k)=\overline{B_{R_{k+1}}}\cap W_{k+1}$ is a compact set in finite dimensional space $W_{k+1}$. Hence, $H$ can be continuously extended to $D_{k+1}$ as an odd function which belongs to $G_{k+1}$. Therefore, by \eqref{4-44},
\begin{equation}\label{4-57}
  b_{k+1}=\inf_{h\in G_{k+1}}\max_{u\in D_{k+1}}E_{2}(h(u))\leq \max_{u\in D_{k+1}}E_{2}(H(u))=E_{2}(H(v_{0}))
\end{equation}
for some $v_{0}\in D_{k+1}$. If $v_{0}\in U_{k}$, by \eqref{4-56} and \eqref{4-57},
\begin{equation}\label{4-58}
  b_{k+1}\leq  E_{2}(H(v_{0}))\leq b_{k}+\varepsilon.
\end{equation}
Suppose $v_{0}\in -U_{k}$. Then, by \eqref{4-11} we obtain
\begin{equation}\label{4-59}
  E_{2}(-H(v_{0}))\geq E_{2}(H(v_{0}))-B_{1}(|E_{2}(H(v_{0}))|^{\frac{1}{q}}+1).
\end{equation}
Since $b_{k}\to \infty$ as $k\to +\infty$ via \eqref{4-45}, it follows from \eqref{4-57} and \eqref{4-59} that
 $ E_{2}(-H(v_{0}))=E_{2}(H(-v_{0}))>0$ for large $k$, e.g. $k\geq k_{1}$ with $k_1\geq \bar{k}$. Thus by \eqref{4-11}, the oddness of $H$, and \eqref{4-56},
\begin{equation}\label{4-60}
\begin{aligned}
  E_{2}(H(v_{0}))&=E_{2}(-H(-v_{0}))\leq E_{2}(H(-v_{0}))+B_{1}\left((E_{2}(H(-v_{0})))^{\frac{1}{q}}+1\right)\\
  &\leq b_{k}+\varepsilon+B_{1}[(b_{k}+\varepsilon)^{\frac{1}{q}}+1]~~\mbox{for}~~k\geq k_{1}.
\end{aligned}
\end{equation}
Therefore, \eqref{4-57}, \eqref{4-58} and \eqref{4-60} yield that
\begin{equation}\label{4-61}
  b_{k+1}\leq b_{k}+\varepsilon+B_{1}[(b_{k}+\varepsilon)^{\frac{1}{q}}+1]~~\mbox{for}~~k\geq k_{1}.
\end{equation}
Since $\varepsilon$ is arbitrary, \eqref{4-61} implies
\begin{equation}\label{4-62}
  b_{k+1}\leq b_{k}+B_{1}(b_{k}^{\frac{1}{q}}+1)\leq b_{k}+2B_{1}b_{k}^{\frac{1}{q}}=b_{k}\left(1+2B_{1}b_{k}^{\frac{1-q}{q}}\right)~~\mbox{for all}~~k\geq k_{2},
\end{equation}
where $k_{2}\geq k_{1}$ such that $b_{k}\geq 1$ for $k\geq k_{2}$. By iteration we obtain that, for $l\in \mathbb{N}^{+}$,
\begin{equation}\label{4-63}
\begin{aligned}
 b_{k_{2}+l}&\leq   b_{k_{2}}\prod_{k=k_{2}}^{k_{2}+l-1}\left(1+2B_{1}\cdot b_{k}^{\frac{1-q}{q}}\right)
 = b_{k_{2}}\exp\left(\sum_{k=k_{2}}^{k_{2}+l-1}\ln\left(1+2B_{1}\cdot b_{k}^{\frac{1-q}{q}} \right) \right)\\
 &\leq b_{k_{2}}\exp\left(2B_{1}\cdot \sum_{k=k_{2}}^{k_{2}+l-1} b_{k}^{\frac{1-q}{q}}\right).
\end{aligned}
\end{equation}
The last step of \eqref{4-63} is derived from the inequality $\ln(1+x)\leq x$ for $x\geq 0$.\par
On the other hand, thanks to Proposition \ref{prop4-4}, we have $b_{k}\geq C_{2}\cdot k^{\frac{2p}{\tilde{\nu}(p-2)}-1}$ for $k\geq k_{2}$. Hence, the $\left(\frac{2p}{\tilde{\nu}(p-2)}-1\right)\cdot \frac{1-q}{q}<-1$ implies there exists a positive constant $M_{2}$ such that
\begin{equation}\label{4-64}
  \sum_{k=k_{2}}^{\infty} b_{k}^{\frac{1-q}{q}}\leq C_{2}^{\frac{1-q}{q}}\sum_{k=k_{2}}^{\infty} \left(k^{\frac{2p}{\tilde{\nu}(p-2)}-1} \right)^{\frac{1-q}{q}}\leq M_{2}<+\infty.
\end{equation}
Therefore, by \eqref{4-63} and \eqref{4-64}, we can find a positive constant $\tilde{M}$ such that $b_{k}\leq \tilde{M}$.

\end{proof}

Now, we finish the proof of Theorem \ref{thm1-2}.

\begin{proof}[Proof of Theorem \ref{thm1-2}]
From Proposition \ref{prop4-4}, we know that there exists $\bar{k}_{1}\geq \tilde{k}$, such that $b_{k}\geq M_{1}\geq M_{0}$ for $k\geq \bar{k}_{1}$. By Proposition \ref{prop4-6}, we can find $l_{1}\geq \bar{k}_{1}$ such that $c_{l_{1}}>b_{l_{1}}$. Then, it follows from Proposition \ref{prop4-3} and Proposition \ref{prop4-5} that there exists $\delta_{1}\in (0,c_{l_{1}}-b_{l_{1}})$ such that $c_{1}=c_{l_{1}}(\delta_{1})\geq c_{l_{1}}>b_{l_{1}}$ is a critical value of $E_{2}$ and is also a critical value of $E_{1}$. Next, by Proposition \ref{prop4-4} again, we can find a $\bar{k}_{2}\geq l_{1}$ such that $b_{k}\geq c_{1}+1$ for $k\geq \bar{k}_{2}$. Similarly, the Proposition  \ref{prop4-3}, \ref{prop4-5} and  \ref{prop4-6} imply that there exists $l_{2}\geq \bar{k}_{2}$ such that $c_{l_{2}}>b_{l_{2}}$. Moreover, there exists $\delta_{2}\in (0, c_{l_{2}}-b_{l_{2}})$ such that $c_{2}=c_{l_{2}}(\delta_{2})\geq c_{l_{2}}>b_{l_{2}}\geq c_{1}+1$ is  another critical value of $E_{1}$. Repeating this process, we can deduce that the functional $E_1$ possesses infinitely many critical points $\{u_m\}_{m=1}^{\infty}$ in $H_{X,0}^1(\Omega)$ such that
\begin{equation}\label{4-65}
 E_1(u_m)=c_m,\qquad 0<c_1<c_2<\cdots<c_m<\cdots,~~~ c_m\rightarrow+\infty,~\mbox{as}~m\rightarrow+\infty.
\end{equation}
Besides, we can easily obtain from \eqref{3-34} and \eqref{4-4} that, there exists a positive constant $\widehat{C}>0$, such that
\begin{equation}\label{4-66}
|E_{1}(u)|\leq \widehat{C}\left(\|u\|_{H_{X,0}^{1}(\Omega)}+\|u\|_{H_{X,0}^{1}(\Omega)}^{2}+\|u\|_{H_{X,0}^{1}(\Omega)}^{p} \right)~~\mbox{for all}~~u\in H_{X,0}^{1}(\Omega).
\end{equation}
Hence,  \eqref{4-65}, \eqref{4-66} and Proposition \ref{prop4-1} imply that $\{u_{m}\}_{m=1}^{\infty}$ is also a sequence of unbounded weak solutions in $H_{X,0}^{1}(\Omega)$.
\end{proof}

\end{document}